\documentclass[12pt,a4paper]{amsart}
	\calclayout
\usepackage[utf8]{inputenc}
\usepackage[T1]{fontenc}
\usepackage{mathtools}
	\numberwithin{equation}{section}
	\mathtoolsset{showonlyrefs}
	\allowdisplaybreaks
\usepackage{hyperref}
	\def\MR#1{\href{http://www.ams.org/mathscinet-getitem?mr=#1}{MR-#1}}
	\def\ARXIV#1{\href{https://arxiv.org/abs/#1}{arXiv:#1}}
\usepackage{mathrsfs}
\usepackage{xcolor}
	\newtheorem{thm}{Theorem}[section]
	\newtheorem{lem}[thm]{Lemma}
	\newtheorem{prop}[thm]{Proposition}
	\newtheorem{cro}[thm]{Corollary}
	\theoremstyle{definition}
	\newtheorem{rem}[thm]{Remark}
	
	\newtheorem{defi}[thm]{Definition}
	\newtheorem{innerasp}{Assumption}
	\newenvironment{asp}[1]{\renewcommand\theinnerasp{#1}\innerasp}{\endinnerasp}
\begin{document}
\title
	[Spine decompositions and limit theorems for critical superprocesses]
	{\large Spine decompositions and limit theorems for a class of critical superprocesses}
\author{Yan-Xia Ren, Renming Song and Zhenyao Sun}
\address
	{Yan-Xia Ren\\
	School of Mathematical Sciences \& Center for
Statistical Science\\
	Peking University\\
	Beijing, P. R. China, 100871}
\email{yxren@math.pku.edu.cn}
\thanks{The research of Yan-Xia Ren is supported in part by NSFC (Grant Nos. 11671017  and 11731009), and LMEQF}
\address
	{Renming Song\\
	Dept of Mathematics\\
	University of Illinois at Urbana-Champaign\\
	Urbana, IL 61801}
\email{rsong@illinois.edu}
\thanks{The research of Renming Song is supported in part by the Simons Foundation (\#429343, Renming Song).}
\address
	{Zhenyao Sun\\
	School of Mathematical Sciences\\
	Peking University\\
	Beijing, P. R. China, 100871}
\curraddr
	{Department of Mathematics\\
	University of Illinois at Urbana-Champaign\\
	Urbana, IL 61801}
\email{zhenyao.sun@pku.edu.cn}
\thanks{Zhenyao Sun is supported by the China Scholarship Council. Corresponding author.}
\begin{abstract}
In this paper we first  establish a decomposition theorem for size-biased Poisson random measures.
As consequences of this decomposition theorem, we get a spine decomposition theorem  and a 2-spine decomposition theorem for some critical superprocesses.
Then we use these spine decomposition theorems to give probabilistic proofs of the
asymptotic behavior of the survival probability and Yaglom's exponential limit law for 
critical superprocesses.
\end{abstract}
\maketitle
\section{Introduction}
\subsection{Motivation}
\label{sec: Motivation}
	It is well known that for a critical Galton-Watson process $\{(Z_n)_{n\in \mathbb N};P\}$, we have
\begin{equation}
\label{eq:Mot-1}
	nP(Z_n>0)
	\xrightarrow[n\to\infty]{} \frac{2}{\sigma^2}
\end{equation}
	and
\begin{equation}
\label{eq:Mot-2}
	\Big\{\frac{Z_n}{n}; P(\cdot|Z_n>0)\Big\}
	\xrightarrow[n\to\infty]{law} \frac{\sigma^2}{2} \mathbf e,
\end{equation}
	where $\sigma^2$ is the variance of the offspring distribution and $\mathbf e$ is an exponential random variable with mean 1.
	The result \eqref{eq:Mot-1} was first proved by Kolmogorov in \cite{Kolmogorov1938Zur-losung} under a third moment condition, and the result \eqref{eq:Mot-2} is due to Yaglom \cite{Yaglom1947Certain}.
	For further references to these results, see \cite{Harris2002The-theory, KestenNeySpitzer1966The-Galton-Watson}.
	Ever since these pioneering papers of Kolmogorov and Yaglom, lots of analogous results
	have been obtained for more general critical branching processes.
	For continuous time critical branching processes, see \cite{AthreyaNey1972Branching}; for discrete time multitype critical branching processes, see \cite{AthreyaNey1972Branching, JoffeSpitzer1967On-multitype}; for continuous time multitype critical branching processes, see \cite{AthreyaNey1974Functionals}; and for critical branching Markov processes, see \cite{AsmussenHering1983Branching}.
	We will call results like \eqref{eq:Mot-1} Kolmogorov type results and results like \eqref{eq:Mot-2} Yaglom type results.
	Similar results have also been obtained for some superprocesses.
	Evans and Perkins \cite{EvansPerkins1990Measure-valued} obtained both Kolmogorov type and Yaglom type results for critical superprocesses when the branching mechanism is $(x,z) \mapsto z^2$ and the spatial motion satisfies some ergodicity conditions.
	Recently, Ren, Song and Zhang \cite{RenSongZhang2015Limit} obtained similar limit results for a class of critical superprocesses with general branching mechanisms and general spatial motions.

	The proofs of the limit results in the papers mentioned above 
	are all analytic in nature and thus not very transparent.
	More intuitive probabilistic proofs would be very helpful.
	This was first accomplished for critical Galton-Watson processes, see \cite{Geiger1999Elementary, LyonsPemantlePeres1995Conceptual} for probabilistic proofs of \eqref{eq:Mot-1}, and \cite{Geiger2000A-new, LyonsPemantlePeres1995Conceptual, RenSongSun2018A-2-spine} for probabilistic proofs of \eqref{eq:Mot-2}.
	For more general models, Vatutin and Dyakonova \cite{VatutinDyakonova2001The-survival} gave a probabilistic proof of a Kolmogorov type result for multitype critical branching processes.
	Recently, Powell \cite{Powell2016An-invariance} gave probabilistic proofs of both Kolmogorov type and Yaglom type results for a class of critical branching diffusions. As far as we know, there is no probabilistic proof of Yaglom type result for multitype critical branching processes, and there are no probabilistic proofs of both Kolmogorov type and Yaglom type results for critical superprocesses yet.

In this paper, we will use the spine method to give probabilistic proofs of both Kolmogorov type and Yaglom type results for a class of critical superprocesses. 
	We will first establish a size-biased decomposition theorem for superprocesses (Theorem \ref{prop:sizBiasDecSupProc}) which will serve as a general framework for the spine method.
	Then, we will establish a spine decomposition theorem for superprocesses (Theorem \ref{prop:sizBiasNMeas}) which is more general than those previously considered in \cite{EckhoffKyprianouWinkel, EnglanderKyprianou2004Local, LiuRenSong2009LlogL}.
	We will also establish a 2-spine decomposition theorem for a class of critical superprocesses (Theorem \ref{prop:2-spine-decomposition}).
	Those spine decompositions are all special forms of
		the aforementioned size-biased decomposition.
	Finally, we use these tools to give probabilistic proofs of a Kolmogorov type result (Theorem \ref{thm:Kolmogorov-type-of-theorem}) and a Yaglom type result (Theorem \ref{thm:Yaglom-type-theorem}) 
for critical superprocesses under slightly weaker conditions than \cite{RenSongZhang2015Limit}.
To develop our decomposition for critical superprocesses, we first prove a size-biased decomposition theorem for Poisson random measures (Theorem \ref{prop:sizBaisPoissRandMeas}),
which we think is of independent interest.
Before we present our main results, we first give a brief review of earlier results on the spine method.
	
The spine method was first introduced in \cite{LyonsPemantlePeres1995Conceptual}.
	Roughly speaking, the spine decomposition theorem says that the size-biased transform of the branching process can be interpreted as an immigration branching process along with an immortal particle.
	This spine approach is generic in the sense that it can be adapted to a variety of general branching processes and is powerful in studying limit behaviors due to its relation with the size-biased transforms.
	In this paper, by the \emph{size-biased transform of a stochastic process} we mean the following:
Suppose that we are given, on some probability space $(\Omega,\mathscr F,P)$, a process $(X_t)_{t\in \Gamma}$, with $\Gamma$ being an arbitrary index set, and a non-negative random variable $G$ with $ P[G] \in (0,\infty)$.
	We say a process $\{(\dot X_t)_{t\in \Gamma}; \dot P\}$ is a \emph{$G$-transform} of the process $\{(X_t)_{t\in \Gamma}; P\}$ if $\{(\dot X_t)_{t\in \Gamma}; \dot P\} \overset{f.d.d.}{=} \{(X_t)_{t\in \Gamma}; P^G\}$, where $P^G$ is a probability  measure on $\Omega$ given by $dP^G := (G/P[G]) dP$. (This also give the definition of a \emph{size-biased transform of a random variable} since a random variable can be considered as a stochastic process 
whose index is a singleton.)

	Using the spine decomposition theorem for the Galton-Watson process $(Z_n)_{n\geq 0}$, Lyons, Pemantle and Peres \cite{LyonsPemantlePeres1995Conceptual} investigated the $Z_n$-transform of the process $(Z_k)_{0\leq k\leq n}$, which is denoted by $(\dot Z_k)_{0\leq k\leq n}$.
	Their key observation in the critical case is that $U\cdot \dot Z_n$ is distributed approximately like $Z_n$ conditioned on $\{Z_n > 0\}$, where $U$ is an independent uniform random variable on $[0,1]$.
	If one denotes by $X$ the weak limit of $\frac{Z_n}{n}$ conditioned on $\{Z_n > 0\}$, and by $\dot X$  the weak limit of $\frac{\dot Z_n}{n}$, then \cite{LyonsPemantlePeres1995Conceptual} proved that $\dot X$ is the $X$-transform of the positive random variable $X$ and
$X\overset{law}{=} U \cdot \dot X,$
	which implies that $X$ is an exponential random variable.

    The spine method is also used by Powell \cite{Powell2016An-invariance}
	to establish results parallel to \eqref{eq:Mot-1} and \eqref{eq:Mot-2} for a class of critical branching diffusion $\{(Y_t)_{t\geq 0}; (P_x)_{x\in D}\}$ in a bounded smooth domain $D\subset \mathbb R^d$.
	As have been discussed in \cite{Powell2016An-invariance}, a direct study of the partial differential equation satisfied by the survival probability $(t,x) \mapsto P_{x}(\|Y_t\| \neq 0)$ is tricky.
	Instead, by using a spine decomposition approach, Powell \cite{Powell2016An-invariance} showed that the survival probability decays like $a(t)\phi(x)$, where $\phi(x)$ is the principal eigenfunction of the mean semigroup of $(Y_t)$ and $a(t)$ is a function capturing the uniform speed.
	In this paper, our proof of the Kolmogorov type result 
   for critical superprocesses follows a similar argument.

The spine method for superprocesses was
	developed in \cite{EckhoffKyprianouWinkel, EnglanderKyprianou2004Local, LiuRenSong2009LlogL} and is very useful in studying limit behaviors of supercritical superprocesses.
	Heuristically, the spine is the trajectory of an immortal moving particle and the spine decomposition theorem says that, after a martingale change of measure, the transformed superprocess can be decomposed in law as an immigration process along this spine.
	The spine decomposition theorem established in this paper is more general than those in \cite{EckhoffKyprianouWinkel, EnglanderKyprianou2004Local, LiuRenSong2009LlogL}.
   We will say more about this in the next subsection.

	Very recently, we developed a 2-spine decomposition technique in \cite{RenSongSun2018A-2-spine} for critical Galton-Watson processes
	and used it to give a new probabilistic proof of Yaglom's result \eqref{eq:Mot-2}.
	One of the facts we used in \cite{RenSongSun2018A-2-spine} is that, if $X$ is a strictly positive random variable with finite second moment,
	then $X$ is an exponential random variable if and only if
\begin{equation}\label{eq:intro-1}
	\ddot X
	\overset{law}{=} \dot X + U \cdot \dot X'
\end{equation}
	where $\dot X$ and $\dot X'$ are independent $X$-transforms of $X$;
	$\ddot X$ is the $X^2$-transform of $X$;
	and $U$ is again an independent uniform random variable on $[0,1]$.
	We then proved in \cite{RenSongSun2018A-2-spine} that the $Z_n(Z_n-1)$-transform of the critical Galton-Watson process $(Z_k)_{0\leq k\leq n}$, which is denoted as $(\ddot Z_k^{(n)})_{0\leq k\leq n}$, can be interpreted as an immigration branching process along a 2-spine skeleton.
	One of those two spines is longer than the other.
	The spirit of our proof in \cite{RenSongSun2018A-2-spine} is to show that the immigration along the longer spine at generation $n$ is distributed approximately like $\dot Z_n$, while the immigration along the shorter spine at generation $n$ is distributed approximately like $\dot Z'_{[U\cdot n]}$.
	Here $\dot Z_n$ and $\dot Z_n'$ are independent $Z_n$-transforms of $Z_n$.
	Roughly speaking, we have
$
	\ddot Z_n^{(n)}
	\overset{law}{\approx} \dot Z_n + \dot Z'_{[U\cdot n]},
$
	and therefore, if $X$ is the weak limit of $\frac{Z_n}{n}$ conditioned on $\{Z_n>0\}$, then $X$ is a positive random variable satisfying \eqref{eq:intro-1}.
In this paper, we adapt the method of \cite{RenSongSun2018A-2-spine} to develop a 2-spine
decomposition for critical superprocesses and then use this 2-spine decomposition to give
probabilistic proofs of Kolmogorov type and Yaglom type results for superprocesses. The spirit of this paper is similar to that of \cite{RenSongSun2018A-2-spine}, but the arguments are more
complicated.

	The idea of multi-spine decomposition is not new. It was first introduced by Harris and Roberts \cite{HarrisRoberts2017The-many} in the context of branching processes.
Our 2-spine methods for Galton-Watson trees \cite{RenSongSun2018A-2-spine} and for superprocesses in this paper are both inspired by \cite{HarrisRoberts2017The-many}.
	An analogous $k$-spine decomposition theorem also appeared in \cite{HarrisJohnstonRoberts2017The-coalescent} and \cite{Johnston2017Coalescence} in the context of continuous time Galton-Watson processes.
	The $k$-th size-biased transform of Galton-Watson trees is also considered in \cite{AbrahamPierre2018Penalization}.
	A closely related infinite spine decomposition is also established in \cite{AbrahamPierre2018Penalization} for the supercritical Galton-Watson tree.
	
	There is another decomposition theorem for supercritical Galton-Watson trees with infinite spines which is first introduced in \cite[Section 12]{AthreyaNey1972Branching} and is now known as the skeleton decomposition.
	The infinite spines in \cite{AbrahamPierre2018Penalization} and the skeleton decomposition 
	in \cite[Section 12]{AthreyaNey1972Branching} are two different decomposition theorems.
	Our 2-spine methods for Galton-Watson trees \cite{RenSongSun2018A-2-spine} and for
    superprocesses in this paper are more relevant to \cite{AbrahamPierre2018Penalization}.

	We mention here that the analog of the skeleton decomposition in \cite[Section 12]{AthreyaNey1972Branching} for supercritical superprocesses is also available and is very popular.
	Heuristically, the skeleton is the trajectories of all the prolific individuals, that is, individuals with infinite lines of descent.
	The skeleton decomposition says that the supercritical superprocess itself can be decomposed in law as an immigration process along this skeleton.
	For the skeleton methods and its applications under a variety of names, see \cite{BerestyckiKyprianouMurillo-Salas2011The-prolific,  BertoinFontbonaMartinez2008On-prolific, DuquesneWinkel2007Growth, EckhoffKyprianouWinkel, EnglanderPinsky1999On-the-construction, EvansOConnell1994Weighted,  KyprianouPerezRen2014The-backbone, KyprianouRen2012Backbone,  Milos2018Spatial,  RenSongZhang2014Central}.
	If we consider critical superprocesses conditioned to be never extinct,
	then we will get the transformed superprocesses (after a Doob's $h$-transformation) considered in \cite{EckhoffKyprianouWinkel, EnglanderKyprianou2004Local, LiuRenSong2009LlogL} for the classical spine decomposition theorem. In this situation, there will be only one prolific individual which is exactly the spine particle.
	So the natural analog of the skeleton decomposition in the critical case is the classical spine decomposition.
	The skeleton decomposition will not be used in this paper.

\subsection{Main results}	
\label{sec: Main results}
	Let $E$ be a locally compact separable metric space.
	We will use $b\mathscr B_E$ and $p\mathscr B_E$ to denote the collection of all bounded Borel functions and positive Borel functions on $E$ respectively.
	We write $bp\mathscr B_E$ for $b\mathscr B_E \cap p\mathscr B_E$.
	For any functions $f,g$ and measure $\mu$ on $E$, we write $\|f\|_\infty := \sup_{x\in E} |f(x)|$, $\mu(f) := \int_E f d\mu$, $\langle \mu,f \rangle := \int_E f d\mu$ and $\langle f,g \rangle_\mu := \int_E fg d\mu$ as long as they have meanings.
	We use $\mathbf 0$ to denote the null measure and use $f\equiv 0$ to mean that $f$ is the zero function.
	If $g(t,x)$ is a function on $[0,\infty)\times E$, we say $g$ is \emph{locally bounded} if $\sup_{t\in [0,T],x\in E} |g(t,x)|<\infty$ for every $T\geq 0$.

	Let the \emph{spatial motion} $\xi=\{(\xi_t)_{t\geq 0};(\mathbb P_x)_{x\in E}\}$ be an $E$-valued Hunt process with its lifetime denoted by $\zeta$ and its transition semigroup denoted by $(P_t)_{t\geq 0}$.
Let the \emph{branching mechanism} $\psi$ be defined as a function on $E\times[0,\infty)$ by
\[
  \psi(x,z) 
  = -\beta(x)z + \alpha(x)z^2+\int_0^\infty (e^{-zr}-1+zr )\pi(x, dr),
  \qquad x\in E, z\geq0,
\]
with $\beta\in b\mathscr B_E,\alpha\in bp\mathscr B_E$ and $\pi(x,dy)$ being a kernel from $E$ to $(0,\infty)$ satisfying that \[\sup_{x\in E} \int_{(0,\infty)} (y\wedge y^2) \pi(x,dy) < \infty.\]
Define an operator $\Psi$ on $p\mathscr B_E$ by
\[
	(\Psi f) (x)
	:= \psi(x,f(x)),
	\quad f\in p\mathscr B_E, x\in E.
\]
Let $\mathcal M_f$ denote the space of all finite measures on $E$ equipped with the weak topology.
A \emph{$(\xi,\psi)$-superprocess} is an $\mathcal M_f$-valued Hunt process
	$X=\{(X_t)_{t\geq 0}; (\mathbf P_\mu)_{\mu \in \mathcal M_f}\}$ satisfying
\begin{equation}
\label{eq: Defi of Vt}
  \mathbf P_\mu [e^{-X_t(f)}] = e^{-\mu(V_tf)},
  \quad t\geq 0, \mu \in \mathcal M_f, f\in bp\mathscr B_E,
\end{equation}
  where, for each $f\in bp\mathscr B_E$, the function $(t,x) \mapsto V_tf(x)$ on $[0,\infty) \times E$ is the unique locally bounded positive solution to the equation
\begin{align}\label{eq:FKPP_in_definition}
  	V_t f(x) + \mathbb P_x \Big[ \int_0^t (\Psi V_{t-s} f)(\xi_s) ds \Big]
	=\mathbb P_x[f(\xi_t)], \quad t \geq 0, x\in E.
\end{align}
	We refer our readers to \cite{Dawson1993Measure-valued, Dynkin1993Superprocesses}
	and \cite[Section 2.3 \& Theorem 5.11]{Li2011Measure-valued}
	for detailed discussions about the existence of such processes.
	Notice that we always have $\mathbf P_{\mathbf 0}(X_t = \mathbf 0) = 1$ for each $t\geq 0$, i.e. the null measure $\mathbf 0$ is an absorption state of the superprocess.

	We will always assume that our superprocess is \emph{non-persistent}:
\begin{asp}{$1$}\label{asp:1}
	$\mathbf P_{\delta_x}(X_t = \mathbf 0) > 0$ for each $x \in E$ and $t>0$.
\end{asp}
	By a \emph{size-biased transform of a measure} we mean the following:
	For a non-negative measurable function $g$ on a measure space $(D,\mathscr F_D,\mathbf D)$ with $\mathbf D(g)\in (0,\infty)$, we define  the \emph{$g$-transform $\mathbf D^g$ of the measure $\mathbf D$} by
\[
	d\mathbf D^g
	:= \frac{g}{\mathbf D(g)} d\mathbf D.
\]
	Note that, the measure $\mathbf D$ is not necessarily a probability measure, but after the $g$-transform, $\mathbf D^g$ is always a probability measure.
	
	Our first result is about a decomposition theorem of the size-biased transforms of superprocesses.
	To state it, we need to introduce the Kuznetsov measures $(\mathbb N_x)_{x\in E}$ (also known as the excursion measures or $\mathbb N$-measures) of the superprocess $X$.
\begin{lem}[{\cite[Section 8.4 \& Theorem 8.24]{Li2011Measure-valued}}]
\label{lem: Kuznetsov measures}
	Under Assumption \ref{asp:1}, there exists an unique family of $\sigma$-finite measures $(\mathbb N_x)_{x\in E}$ defined on the Skorokhod space of measure-valued paths
\[
  \mathcal W :=\{ w= (w_t)_{t\geq 0}: w \text{ is an $\mathcal M_f$-valued c\`{a}dl\`{a}g function on $[0,\infty)$ having $\mathbf 0$ as a trap}\}
\]
	such that
\begin{enumerate}
\item
    $\mathbb N_x \{\forall t > 0, w_t=\mathbf 0\} =0$ for each $x\in E$;
\item
    $\mathbb N_x\{w_0 \neq \mathbf 0\} = 0$ for each $x\in E$;
\item
    for each $\mu \in \mathcal M_f$, if $\mathcal N(dw)$ is a Poisson random measure on $\mathcal W$ with mean measure
\[
  \mathbb N_\mu(dw):= \int_E \mathbb N_x(dw)\mu(dx), \quad w\in \mathcal W,
\]
  then the process defined by
\[
	\widetilde X_0 := \mu;
	\quad \widetilde X_t
	:=\int_{\mathcal W}w_t~\mathcal N(dw),
	\quad t>0,
\]
	is a realization of the superprocess $\{X; \mathbf P_{\mu}\}$.
\end{enumerate}
\end{lem}
The measures $(\mathbb N_x)_{x\in E}$ are called the \emph{Kuznetsov measures of the superprocess $X$}.
		Note that, the superprocess $X$ itself can be considered as a $\mathcal W$-valued random element.
    Roughly speaking, the branching property of superprocess  says that $X$ can be considered as an ``infinitely divisible'' $\mathcal W$-valued random element.
		The Kuznetsov measure $\mathbb N_x$ can then be interpreted as the ``L\'{e}vy measure'' of $X$ under $\mathbf P_{\delta_x}$.
	We refer our readers to \cite{DynkinKuznetsov2004N-measure} and \cite[Section 8.4]{Li2011Measure-valued} for more details about such measures.

In the remainder of this paper, we will always use $(\mathbb N_x)_{x\in E}$ to denote the
Kuznetsov measures of our superprocess $X$.
We will always use $w = (w_t)_{t\geq 0}$ to denote a generic element in $\mathcal W$.
With a slight abuse of notation,
we always assume that our superprocess $X$ is given by
\[
	X_0 := \mu;
	\quad X_t
	:=\int_{\mathcal W}w_t~\mathcal N(dw),
	\quad t>0,
\]
	where, for each $\mu \in \mathcal M_f$, $\{\mathcal N; \mathbf P_\mu\}$ is a Poisson random measure on $\mathcal W$ with mean measure $\mathbb N_\mu$.
Recall that, for any $w\in \mathcal W$ and $t\geq 0$, $w_t$ is a finite measure on $E$, and thus $w_t(f)=\int_Ef(x)w_t(dx)$ for any $f \in p\mathscr B_E$.

Our first result is about the $\mathcal N(F)$-transform of the superprocess $X$, where $F$ is a non-negative measurable function on $\mathcal W$ with $\mathbb N_\mu[F] \in (0,\infty)$ for a given $\mu \in \mathcal M_f$. In this case, according to Campbell's formula, we have
\[
	\mathbf P_\mu[\mathcal N(F)]= \mathbb N_\mu[F] \in (0,\infty).
\]	
	Therefore, both $\mathbb N_\mu^F$ --- the $F$-transform of $\mathbb N_\mu$, and $\mathbf P_\mu^{\mathcal N(F)}$ --- the $\mathcal N(F)$-transform of $\mathbf P_\mu$, are well defined probability measures.
\begin{thm}\label{prop:sizBiasDecSupProc}
	Suppose that Assumption \ref{asp:1} holds.
Let $\mu \in \mathcal M_f$ and $F$ be a non-negative measurable function on $\mathcal W$ with $\mathbb N_\mu(F)\in (0,\infty)$ .
	Let $\{(Y_t)_{t\geq 0}; \mathbf Q_\mu\}$ be a $\mathcal W$-valued random element with
 law $\mathbb N^F_\mu$.
	Then we have
$
	\{(X_t)_{t\geq 0}; \mathbf P_\mu^{\mathcal N(F)}\}
	\overset{f.d.d.}{=} \{(X_t + Y_t)_{t\geq 0}; \mathbf P_\mu \otimes \mathbf Q_\mu\}.
$
\end{thm}
In order to prove Theorem \ref{prop:sizBiasDecSupProc},
	we develop a decomposition theorem for size-biased transforms of Poisson random measures which we think should be of independent interest:
\begin{thm}\label{prop:sizBaisPoissRandMeas}
	Let $(S, \mathscr S)$ be a measurable space with a $\sigma$-finite measure $N$.
	Let $\{\mathbf N; P\}$ be a Poisson random measure on $(S, \mathscr S)$
	with mean measure $N$.
	Let $g \in p \mathscr S$ satisfy $N(g)\in (0,\infty).$
	Denote by $N^g$ and $P^{\mathbf N(g)}$ the $g$-transform of $N$ and the $\mathbf N(g)$-transform of $P$, respectively.
	Let $\{\vartheta;Q\}$ be an $S$-valued random element
	with law $N^g$.
	Then we have
$
	\{\mathbf N;P^{\mathbf N(g)}\}
	\overset{law}{=} \{\mathbf N + \delta_\vartheta;P\otimes Q\}.
$
\end{thm}
	Define $(S_t)_{t \geq 0}$ the \emph{mean semigroup} of the superprocess $X$ by
\begin{equation}
\label{eq: mean semigroup}
	S_t f(x)
	:= \mathbb P_x[e^{\int_0^t \beta(\xi_s) ds} f(\xi_t)],
	\quad x \in E,t \geq 0,f \in p\mathscr B_E.
\end{equation}
	For each $\mu \in \mathcal M_f$, we define $ (\mu\mathbb P)(\cdot):= \int_E \mathbb P_x(\cdot)\mu(dx)$.
	Note that $\mu\mathbb P$ is not necessarily a probability measure.
	It is well known (see \cite[Proposition 2.27]{Li2011Measure-valued} for example) that for each $\mu \in \mathcal M_f, t \geq 0$ and $f \in p\mathscr B_E,$
\begin{equation}
\label{eq: mean formula}
	\mathbf P_\mu[X_t(f)]
	= \mathbb N_\mu[w_t(f)]
	= (\mu\mathbb P)[e^{\int_0^T \beta(\xi_s)ds}f(\xi_T)\mathbf 1_{T<\zeta}]
	=\mu(S_t f).
\end{equation}

	Thanks to Theorem \ref{prop:sizBiasDecSupProc}, in order to study the size-biased transform of a superprocess we only have to study the corresponding size-biased transform of its Kuznetsov measures. 	
	We first consider the case when the function $F$ in Theorem \ref{prop:sizBiasDecSupProc} takes the form of $F(w)=w_T(g)$ where $T>0$ and $g\in p\mathscr B_E$ with $\mu(S_Tg)\in (0,\infty)$ for a given $\mu \in \mathcal M_f$.
	In this case, according to \eqref{eq: mean formula}, we have
\[
	\mathbf P_\mu [X_T(g)]
	= \mathbb N_\mu[w_T(g)] 
	= (\mu\mathbb P)[e^{\int_0^T \beta(\xi_s)ds}g(\xi_T)\mathbf 1_{T < \zeta}]
	\in (0,\infty).
\]
	Therefore, $\mathbf P_\mu^{X_T(g)}$ --- the $X_T(g)$-transform of $\mathbf P_\mu$, $\mathbb N_\mu^{w_T(g)}$ --- the $w_T(g)$-transform of the Kuznetsov measure $\mathbb N_\mu$, 
	and $\mathbb P^{(g,T)}_\mu$ --- the $(e^{\int_0^T\beta(\xi_s)ds} g(\xi_T)\mathbf 1_{T< \zeta})$-transform of the measure $\mu\mathbb P$, 
	are all well defined probability measures.
	Also note that, in this case, we have $X_T(g) = \mathcal N(F)$, therefore $\mathbf P_\mu^{X_T(g)} = \mathbf P_\mu^{\mathcal N(F)}$.
Recall that the superprocess $X$ itself can be considered as a $\mathcal W$-valued random element.
	Denote by $\mathbf P_\mu(X \in dw)$ the push-forward of $\mathbf P_\mu$ under $X$, i.e., the distribution of $X$ under $\mathbf P_\mu$. Then, $\mathbf P_\mu(X \in dw)$ is a probability measure on $\mathcal W$.
	Recall that we always assume that Assumption \ref{asp:1} holds.
\begin{defi}
\label{def: Spine representation}
Suppose that $\mu \in \mathcal M_f$, $T >0$ and $g \in p\mathscr B_E$ satisfy $\mu(S_Tg)\in (0,\infty)$.
	We say $\{(\xi_t)_{0\leq t\leq T}, (Y_t)_{0\leq t\leq T}, \mathbf n_T; \dot {\mathbf P}^{(g,T)}_\mu\}$ is a \emph{spine representation of $\mathbb N_\mu^{w_T(g)}$} if the following are true:
\begin{enumerate}
\item \label{def: Spine representation 1}
	The \emph{spine process} $\{(\xi_t)_{0\leq t\leq T}; \dot{\mathbf P}^{(g,T)}_\mu\}$ is a copy of $\{(\xi_t)_{0\leq t\leq T}; \mathbb P^{(g,T)}_\mu\}$.
\item
Conditioned on $\sigma(\xi_t: 0 \leq t\leq T)$, the \emph{immigration process}
	$\{(Y_t)_{0\leq t\leq T}; \dot{\mathbf P}^{(g,T)}_\mu\}$ is an $\mathcal M_f$-valued process given by
\begin{equation}\label{eq:defSpinImmigr}
	Y_t
	:= \int_{(0,t] \times \mathcal W} w_{t-s} \mathbf n_T(ds,dw),
	\quad 0 \leq t\leq T,
\end{equation}
	where,
	$\mathbf n_T$
	is a Poisson random measure on $[0,T] \times \mathcal W$ with mean measure
\begin{equation}\label{eq:meanMeasImmigr}
	\mathbf m^\xi_T(ds,dw)
	:= 2 \alpha(\xi_s) \mathbb N_{\xi_s}(dw)\cdot ds +  \int_{(0,\infty)} y \mathbf P_{y\delta_{\xi_s}}(X\in dw) \pi(\xi_s,dy)\cdot ds.
\end{equation}
\end{enumerate}
\end{defi}
	We are now ready to present our theorem on the spine decomposition of superprocesses:
\begin{thm}\label{prop:sizBiasNMeas}
	Suppose that Assumption \ref{asp:1} holds.
	Suppose that $\mu \in \mathcal M_f$, $T >0$ and $g \in p\mathscr B_E$ satisfy $\mu(S_Tg)\in (0,\infty)$.
	Let $\{(\xi_t)_{0\leq t\leq T}, (Y_t)_{0\leq t\leq T}, \mathbf n_T; \dot {\mathbf P}^{(g,T)}_\mu\}$ be a spine representation of $\mathbb N_\mu^{w_T(g)}$.
	Then,
$
	\{(Y_t)_{t\leq T}; \dot{\mathbf P}^{(g,T)}_\mu\}
	\overset{f.d.d.}{=} \{(w_t)_{t\leq T}; \mathbb N_\mu^{w_T(g)}\}.
$
\end{thm}
	As a simple consequence of Theorems \ref{prop:sizBiasDecSupProc} and  \ref{prop:sizBiasNMeas}, we have the following:
\begin{cro}\label{cro: spine decomposition}
	Suppose that Assumption \ref{asp:1} holds.
	Suppose that $\mu \in \mathcal M_f$, $T >0$ and $g \in p\mathscr B_E$ satisfy $\mu(S_Tg)\in (0,\infty)$.
	Let $\{(\xi_t)_{0\leq t\leq T}, (Y_t)_{0\leq t\leq T}, \mathbf n_T; \dot {\mathbf P}^{(g,T)}_\mu\}$ be a spine representation of $\mathbb N_\mu^{w_T(g)}$.
	Then,
$
	\{(X_t)_{t\geq 0}; \mathbf P_\mu^{X_T(g)}\}
	\overset{f.d.d.}{=} \{(X_t + Y_t)_{t\geq 0}; \mathbf P_\mu \otimes \dot {\mathbf P}^{(g,T)}_\mu\}.
$
\end{cro}
	
	Corollary \ref{cro: spine decomposition} can be considered as a generalization of the classical spine decomposition theorem for superprocesses developed in \cite{EckhoffKyprianouWinkel, EnglanderKyprianou2004Local, LiuRenSong2009LlogL}.
		In these earlier papers, the testing function $g$ is chosen specifically to be  the principal eigenfunction $\phi$ of the mean semigroup of the superprocess (which will be introduced shortly).
In the classical case (i.e. $g = \phi$), the four families of probability measures
	$(\mathbf P_\mu^{X_T(g)})_{T\geq 0}$, $(\mathbb P_\mu^{(g, T)})_{T\geq 0}$, $(\dot{\mathbf P}_{\mu}^{(g,T)})_{T>0}$
		and $(\mathbb N_\mu^{w_T(g)})_{T> 0}$ are all consistent,
	but  in the general case ( i.e. $g\neq \phi$), they are typically not consistent.
	More details about these consistencies will be provided in Lemma \ref{lem: measure of spine} and Remark \ref{rem: consistency}.
	
In the papers mentioned in the paragraph above, the Kuznetsov measures have already been used to describe infinitesimal immigrations along the spine.
	However, our Theorem \ref{prop:sizBiasNMeas} provides
   another  relation between
	immigration and the Kuznetsov measures: the total immigration $\{(Y_t)_{t\geq 0}; \dot {\mathbf P}^{(g,T)}_\mu\}$ actually has law of a size-biased transform of the Kuznetsov measures.
It seems that this fact has not been exploited before, even in the classical case.

	The study of the limit behavior of superprocesses $X$
		relies heavily on the spectral property of the mean semigroup.
 In this paper, we assume the following:
\begin{asp}{$2$}\label{asp:2}
	There exist a $\sigma$-finite Borel measure $m$ with full support on $E$ and a family of strictly positive, bounded continuous functions $\{ p(t,\cdot,\cdot): t > 0 \}$ on $E \times E$ such that,
\begin{align}
	P_t f(x)
	= \int_E p(t,x,y) f(y) m(dy),
	&\quad t>0, x \in E,f \in b\mathscr B_E,\label{eq:kernMeanSemGroup}
	\\ \int_E p(t,x,y)m(dx)
	\leq 1,
	&\quad t>0,y\in E,\label{eq:kernMeanSemGroup2}
	\\ \int_E \int_E p(t,x,y)^2 m(dx) m(dy)
	<\infty,
	&\quad t> 0,\label{eq:kernMeanSemGroup3}
\end{align}
	and that $x \mapsto \int_E p(t,x,y)^2 m(dy)$ and $y \mapsto \int_E p(t,x,y)^2 m(dx)$ are both continuous on $E$.
\end{asp}

	In the reminder of this paper, we will always use $m$ to denote the reference measure in Assumption \ref{asp:2}.

	Assumption \ref{asp:2} is a pretty weak assumption. \eqref{eq:kernMeanSemGroup2} implies that the adjoint operator $P^*_t$
of $P_t$ is also Markovian, and  \eqref{eq:kernMeanSemGroup3} implies that
$P_t$ and $P^*_t$ are Hilbert-Schmidt operators.
	Under Assumption \ref{asp:2}, it is proved in \cite{RenSongZhang2015Limit} and \cite{RenSongZhang2017Central} that the semigroup $(P_t)_{t \geq 0}$ and its adjoint semigroup $(P^*_t)_{t \geq 0}$ are both strongly continuous semigroups of compact operators on $L^2(E,m)$.
	According to \cite[Lemma 2.1]{RenSongZhang2015Limit}, there exists a function $q(t,x,y)$ on $(0,\infty) \times E \times E$ which is continuous in $(x,y)$ for each $t>0$ such that
\[
	e^{-\|\beta\|_\infty t} p(t,x,y)
	\leq q(t,x,y)
	\leq e^{\|\beta\|_\infty t} p(t,x,y),
	\quad t>0, x, y\in E,
\]
	and that for any $t>0, x\in E$ and $f \in b\mathscr B_E$,
\begin{equation}\label{eq: density of S}
	S_t f(x)
	= \int_E q(t,x,y) f(y) m(dy).
\end{equation}
	(From \eqref{eq: mean formula}, we see that $q(t,x,y) m(dy)$ can be
	roughly interpreted as the density of the expected mass of $X_t$ at position $y$,
	under probability $\mathbf P_{\delta_x}$.)
	Define a family of transition kernels $(S^*_t)_{t \geq 0}$ on $E$ by
\[
	S^*_0 = I;
	\quad S^*_t f(y)
	:= \int_E q(t,x,y) f(x) m(dx),
	\quad t>0, y\in E, f\in b\mathscr B_E.
\]
	It is clear that $(S^*_t)_{t \geq 0}$ is the adjoint semigroup of $(S_t)_{t \geq 0}$ in $L^2(E,m)$.
	It is proved in \cite{RenSongZhang2015Limit} and \cite{RenSongZhang2017Central} that $(S_t)_{t \geq 0}$ and $(S^*_t)_{t \geq 0}$ are also strongly continuous semigroups of compact operators in $L^2(E,m)$.
	Let $L$ and $L^*$ be the generators of the semigroups $(S_t)_{t \geq 0}$ and $(S^*_t)_{t \geq 0}$, respectively.
	Denote by $\sigma(L)$ and $\sigma(L^*)$ the spectra of $L$ and $L^*$, respectively.
	According to \cite[Theorem V.6.6.]{Schaefer1974Banach}, $\lambda := \sup \text{Re}(\sigma(L)) = \sup \text{Re}(\sigma(L^*))$ is a common eigenvalue of multiplicity $1$ for both $L$ and $L^*$.
	Using the argument in \cite{RenSongZhang2015Limit}, the eigenfunctions $\phi$ of $L$ and $\phi^*$ of $L^*$ associated with the eigenvalue $\lambda$ can be chosen to be strictly positive and continuous everywhere on $E$.
	We further normalize $\phi$ and $\phi^*$ so that $\langle\phi, \phi\rangle_m = \langle\phi,\phi^*\rangle_m = 1$.
	Moreover, for each $t\geq 0,x\in E$, we have $S_t \phi(x) = e^{\lambda t} \phi(x)$ and $S^*_t \phi^*(x) = e^{\lambda t} \phi^*(x)$.
	We call $\phi$ the \emph{principal eigenfunction} of the mean semigroup $(S_t)_{t\geq 0}$.
\begin{rem}
	Note that we do not require the operators $(P_t)_{t\geq 0}$ to be self-adjoint in $L^2(E,m)$, i.e., we do not assume $p(t,x,y)= p(t,y,x)$ for each $x,y\in E$ and $t>0$. In other word, the spatial motion $\xi$ considered in this paper is not necessarily a symmetric Markov process with respect to the measure $m$.
	As a consequence, $(S_t)_{t\geq 0}$ are not necessarily self-adjoint either.
	\end{rem}

We will use the following function
 \[
	A(x)
	:= 2\alpha(x) + \int_{(0,\infty)} y^2\pi(x,dy),
	\quad x\in E
  \]
in Assumption \ref{asp:3} below.

	For all $t \geq 0$ and $x\in E$, it is now clear that
$
  \mathbf P_{\delta_x}[X_t(\phi)]
  = S_t \phi(x)
  = e^{\lambda t} \phi(x).
$
  If $\lambda > 0$, the mean of $X_t(\phi)$ will increase exponentially; if $\lambda < 0$, the mean of $X_t(\phi)$ will decrease exponentially; and if $\lambda = 0$, the mean of $X_t(\phi)$ will be a constant.
  Because of this, we say $X$ is \emph{supercritical, critical} or \emph{subcritical}, according to $\lambda > 0$, $\lambda = 0$ or $\lambda < 0$, respectively.
	In this paper, we are mainly interested in critical superprocesses with finite second moments.
So, for the remainder of this paper, we always assume the following:
\begin{asp}{$3$}\label{asp:3}
\begin{enumerate}
\item \label{asp:3 1} 
	The superprocess $X$ is critical, i.e., $\lambda = 0 $.
\item \label{asp:3 2} 
	The function $\phi A:x\mapsto \phi(x)A(x)$  is bounded on $E$.
\end{enumerate}
\end{asp}

	Assumption \ref{asp:3}.\eqref{asp:3 2} is satisfied, for example, when $\phi$ and $A$ are bounded on $E$.
	These conditions appeared in the literature and was used by \cite{RenSongZhang2015Limit} in the proof of the Kolmogorov type and the Yaglom type results for critical superprocesses.

 	Denote by $\mathcal M_f^\phi$ the collection of all the measures $\mu \in \mathcal M_f$ such that $\mu(\phi) \in (0,\infty)$.
  	It will be proved in Proposition \ref{prop:covanrance} that $\mathbf P_\mu[X_t(\phi)^2]< \infty$ for each $\mu \in \mathcal M^\phi_f$ and $t>0$ provided the function
$
	\phi A: x\mapsto \phi(x) A(x)
$
	is bounded on $E$.

	Taking $\mu \in \mathcal M_f^{\phi}$, $T\geq 0$ and $g = \phi$ in Definition \ref{def: Spine representation}.\eqref{def: Spine representation 1}, it will be proved in Lemma \ref{lem: measure of spine} that the family of probability measures $(\mathbb P_\mu^{(\phi,T)})_{T\geq 0}$ is consistent,
	i.e., there exists an $E$-valued process $\{(\xi_t)_{t\geq 0}; \dot{\mathbb P}_\mu\}$ such that
\[	\{(\xi_t)_{0\leq t\leq T}; \mathbb P_\mu^{(\phi,T)}\}
	\overset{f.d.d}{=}\{(\xi_t)_{0\leq t\leq T}; \dot{\mathbb P}_\mu\},
	\quad T\geq 0.
\]	
	The process $\{(\xi_t)_{t\geq 0}; \dot{\mathbb P}_\mu\}$ is exactly the spine process in the classical spine decomposition.

It will also be proved in Proposition \ref{prop:covanrance} that, under Assumptions \ref{asp:1}, \ref{asp:2} and \ref{asp:3}, for all $\mu \in \mathcal M_f^\phi$ and $T>0$, we have
\[
	\mathbb N_\mu[w_T(\phi)^2]
	= \langle\mu, \phi\rangle \dot{\mathbb P}_{\mu} \Big[ \int_0^T (A\phi)(\xi_s) ds\Big]\in (0,\infty).
\]
	As a consequence, $\mathbb N_\mu^{w_T(\phi)^2}$ --- the $w_T(\phi)^2$-transform of $\mathbb N_\mu$, and $\ddot{\mathbb P}^{(T)}_\mu$ --- the $(\int_0^T (A\phi)(\xi_s) ds)$-transform of $\dot{\mathbb P}_{\mu}$,  are both well defined probability measures.
Recall that we always assume that Assumptions \ref{asp:1}, \ref{asp:2} and \ref{asp:3} hold.
\begin{defi}
	Let $\mu \in \mathcal M_f^\phi$ and $T>0$.
	We say
\[
	\{(\xi_t)_{0\leq t\leq T}, \kappa, (\xi_t')_{\kappa\leq t\leq T}, (Y_t)_{0\leq t\leq T}, \mathbf n_T,(Y'_t)_{\kappa \leq t\leq T}, \mathbf n'_T, (
	X'_t)_{\kappa \leq t\leq T}, (Z_t)_{0\leq t\leq T} ; \ddot {\mathbf P}_\mu^{(T)}\}
\]
	is a \emph{2-spine representation} of $\mathbb N_\mu^{w_T(\phi)^2}$ if the following are true:
\begin{enumerate}
\item
	\emph{The main spine} $\{(\xi_t)_{0\leq t\leq T}; \ddot{\mathbf P}_\mu^{(T)}\}$ is a copy of $\{(\xi_t)_{0\leq t\leq T}; \ddot{\mathbb P}_{\mu}^{(T)}\}$.
\item
	Conditioned on $(\xi_t)_{0\leq t \leq T}$, the \emph{splitting time} $\kappa$ is a random variable taking values in $[0,T]$ with law
\[
	\ddot{\mathbf P}_\mu^{(T)}\big(\kappa \in ds\big|(\xi_t)_{0\leq t\leq T}\big)
	=\frac {\mathbf 1_{0\leq s\leq T} (A\phi)(\xi_s) ds} {\int_0^T (A\phi)(\xi_r) dr}.
\]
\item
	Conditioned on $(\xi_t)_{t \leq T}$ and $\kappa$, the \emph{auxiliary spine} $(\xi'_t)_{\kappa \leq t \leq T}$ is defined such that
\begin{equation}\label{eq:defAuxilSpin}
	\{(\xi'_{\kappa+t})_{0 \leq t \leq T - \kappa}; \ddot{\mathbf P}_\mu^{(T)}(\cdot | \xi,\kappa) \}
	\overset{law}{=} \{(\xi_t)_{0 \leq t \leq T - \kappa}; \dot{\mathbb P}_{\xi_\kappa} \}.
\end{equation}
\item
	Write $\mathscr G
	:= \sigma \{ (\xi_t)_{t \leq T}, \kappa, (\xi_t')_{\kappa \leq t \leq T} \}$.
	Conditioned on $\mathscr G$, \emph{the main immigration} $(Y_t)_{0 \leq t\leq T}$ is given by
\[
	Y_t
	:= \int_{(0,t] \times \mathcal W} w_{t-s} \mathbf n_T(ds, dw),
	\quad t\in [0,T],
\]
	where $\mathbf n_T$ is a
	Poisson random measure on $[0,T] \times \mathcal W$ with mean measure
\[
	\mathbf m_T^\xi (ds,dw)
:=  2 \alpha(\xi_s)  \mathbb N_{\xi_s}(dw)\cdot ds
 + \int_{(0,\infty)} y \mathbf P_{y\delta_{\xi_s}}(X\in dw) \pi(\xi_s,dy)\cdot ds.
\]
\item
	Conditioned on $\mathscr G$, \emph{the auxiliary immigration} $(Y'_t)_{\kappa \leq t \leq T}$ is given by
\[
	Y'_t
	:= \int_{(\kappa,t] \times \mathcal W} w_{t-s} \mathbf n'_T(ds,dw),
	\quad t \in [\kappa,T],
\]
	where $\mathbf n'_T$ is a
	Poisson random measure on $[\kappa,T] \times \mathcal W$ with mean measure
\[
	\mathbf m^{\xi'}_{\kappa,T}(ds,dw)
:=  2 \alpha(\xi'_s)  \mathbb N_{\xi'_s}(dw)\cdot ds
+  \int_{(0,\infty)} y \mathbf P_{y\delta_{\xi'_s}}(X\in dw) \pi(\xi'_s,dy)\cdot ds.
\]
\item
	Conditioned on $\mathscr G$, \emph{the splitting-time immigration}
$(X'_t)_{\kappa \leq t \leq T}$ is defined by
\[
	\{(X'_{\kappa+t})_{0\leq t\leq T-\kappa}; \ddot{\mathbf P}_\mu(\cdot | \mathscr G)\}
	\overset{law}{=} \{(X_t)_{0 \leq t \leq T-\kappa}; \widetilde{\mathbf P}_{\xi_\kappa}\},
\]
where, for each $x\in E$, the probability measure $\widetilde{\mathbf P}_{x}$ is given by
\begin{equation}\label{eq:def-tilde-P}
	\widetilde{\mathbf P}_{x}(\cdot)
	:=
\begin{cases}
	\frac{2\alpha(x) \mathbf P_{\mathbf 0}(\cdot)+\int_{(0,\infty)}y^2\mathbf P_{y\delta_x}(\cdot)\pi(x,dy)}{2\alpha(x)+\int_{(0,\infty)}y^2\pi(x,dy)},
	&\quad \mbox{if } A(x)>0,\\
	\mathbf P_{\mathbf 0}(\cdot),
	&\quad \mbox{if } A(x)=0.
\end{cases}
\end{equation}
\item
	Conditioned on $\mathscr G$, the main immigration
	$\{Y,\mathbf n_T\}$, the auxiliary immigration $\{Y',\mathbf n'_T\}$ and the splitting-time immigration $X'$ are mutually independent.
	Setting $Y'_t = \mathbf 0$ and $X'_t = \mathbf 0$ for each $t\leq \kappa$, the \emph{total immigration $(Z_t)_{0\leq t\leq T}$} is given by
\[
	Z_t
	:= Y_t + Y_t' + X_t',
	\quad 0\leq t\leq T.
\]
\end{enumerate}
\end{defi}

We are now ready to state our 2-spine decomposition theorem for critical superprocesses:

\begin{thm}\label{prop:2-spine-decomposition}
	Suppose that Assumptions \ref{asp:1}, \ref{asp:2} and \ref{asp:3} hold.
	Let $\mu\in\mathcal M_f^\phi$ and $T>0$.
	Suppose that
$
	\{(\xi_t)_{0\leq t\leq T}, \kappa, (\xi_t')_{\kappa\leq t\leq T}, (Y_t)_{0\leq t\leq T}, \mathbf n_T,(Y'_t)_{\kappa \leq t\leq T}, \mathbf n'_T, (
	X'_t)_{\kappa \leq t\leq T}, (Z_t)_{0\leq t\leq T} ; \ddot {\mathbf P}_\mu^{(T)}\}
$
	is a 2-spine representation of $\mathbb N_\mu^{w_T(\phi)^2}$.
	Then
$
	\{(Z_t)_{t\leq T}; \ddot {\mathbf P}^{(T)}_\mu\}
	\overset{f.d.d.}{=} \{(w_t)_{t\leq T}; \mathbb N^{w_T(\phi)^2}_\mu\}.
$
\end{thm}
As mentioned earlier in Subsection \ref{sec: Motivation},
	this 2-spine decomposition theorem for superprocesses is an analog of the 2-spine decomposition theorem for Galton-Watson trees in \cite{RenSongSun2018A-2-spine}, and is closely related to the multi-spine theory appeared in  \cite{HarrisRoberts2017The-many}, \cite{HarrisJohnstonRoberts2017The-coalescent}, \cite{Johnston2017Coalescence} and \cite{AbrahamPierre2018Penalization}.
Of course, depend on the choice of $F$, there are many versions of Theorem \ref{prop:sizBiasDecSupProc}.
	We only consider the cases when $F(w)$ takes the forms of $w_t(g)$ and $w_t(\phi)^2$, because they are sufficient for our purpose to give probabilistic proofs of the Kolmogorov type and Yaglom type results for critical superprocesses.

	We now turn our attention to the limit behavior of critical superprocesses. First, we want to consider the asymptotic behavior of $v_t(x):= - \log \mathbf P_{\delta_x}(X_t = \mathbf 0)$, where $t>0$ and $x\in E$. (They are well defined thanks to Assumption \ref{asp:1}.)
	From \eqref{eq: Defi of Vt} and monotone convergence, we have
\begin{equation}
\label{eq: defi of vt}
	v_t(x) = \lim_{\theta \to \infty}V_t(\theta \mathbf 1_E)(x),
	\quad t> 0, x\in E,
\end{equation}
and
\begin{equation}\label{eq: extinction probability with vt}
	\mathbf P_{\mu}(X_t = \mathbf 0) = e^{- \mu(v_t) },
	\quad \mu \in \mathcal M_f, t\geq 0,
\end{equation}
 where the operators $(V_t)_{t\geq 0}$ are given by \eqref{eq: Defi of Vt}.
	We call $(V_t)_{t \geq 0}$ the \emph{cumulant semigroup} of the superprocess $X$, because it satisfies the semigroup property in the sense that, for all $f\in p\mathscr B_E, t, s \geq 0$ and $x \in E$, it holds that $V_t V_sf(x) = V_{t+s} f(x)$ (see \cite[Theorem 2.21]{Li2011Measure-valued}).

	Let $\psi_0$ be a function on $E\times[0,\infty)$ defined by
\[
	\psi_0(x,z)
	:= \psi(x,z) + \beta(x)z
	= \alpha(x)z^2 + \int_{(0,\infty)}(e^{-rz}-1+rz) \pi(x,dr),
	\quad x\in E, z\geq 0.
\] 
	Let $\Psi_0$ be an operator on $p\mathscr B_E$ defined by
\[
	(\Psi_0 f)(x)
	:= \psi_0(x,f(x)),
	\quad f\in p\mathscr B_E, x\in E.
\]
	It is known, see \cite[Theorem 2.23]{Li2011Measure-valued} for example, that for each $f\in bp\mathscr B_E$, $(t,x) \mapsto V_tf(x)$ is the solution of the equation
\begin{equation}
\label{eq:mean-fkpp}
	V_t f(x) + \int_0^t (S_{t-s}\Psi_0 V_s f)(x) ds
	= S_t f(x),\quad t\geq 0, x\in E.
\end{equation}
	Indeed, \eqref{eq:mean-fkpp} can be obtained from \eqref{eq:FKPP_in_definition} using a Feynman–Kac type  argument.
	It is also clear that
\begin{equation}\begin{split}\label{eq: simigroup for small vt}
	V_t v_s(x)
	&= -\log \mathbf P_{\delta_x}[e^{-\langle X_t,\lim_{\theta\to\infty } V_s(\theta \mathbf 1_E)\rangle}]
	= -\lim_{\theta \to \infty} \log \mathbf P_{\delta_x}[e^{-\langle X_t, V_s(\theta \mathbf 1_E)\rangle}]\\
	&= - \lim_{\theta\to\infty} V_t V_s(\theta \mathbf 1_E)(x)
	= v_{t+s}(x),
	\quad s,t>0, x\in E.
\end{split}\end{equation}
	So, if we allow extended values, it follows from  \eqref{eq:mean-fkpp} and \eqref{eq: simigroup for small vt} that we have the following equation for $(v_t)_{t\geq 0}$:
\begin{equation}\label{eq:reason-for-asp2'}
	v_{t+s}(x) + \int_0^t  (S_{t-r} \Psi_0 v_{r+s})(x)  dr
	= S_tv_s(x),
	\quad x\in E,t\geq 0.
\end{equation}
	 In order to study the asymptotic behavior of $(v_t)_{t\geq 0}$ using \eqref{eq:reason-for-asp2'},
	we need to understand the asymptotic behavior of the mean semigroup $(S_t)_{t\geq 0}$. The following assumption is commonly used for this purpose:
\begin{asp}{$2'$}
\label{asp:2'}
	In addition to Assumption \ref{asp:2}, we further assume that the mean semigroup $(S_t)_{t\geq 0}$ is \emph{intrinsically ultracontractive},
	that is, for each $t>0$ there exists $c_t>0$ such that for all $x,y\in E$, we have $q(t,x,y)\leq c_t\phi(x)\phi^*(y)$.
\end{asp}

		The concept of intrinsic ultracontractivity was first introduced by Davies and Simon \cite{DS} in the symmetric setting and was extended to the non-symmetric setting in \cite{KimSong2008Intrinsic}.  Assumption 	\ref{asp:2'} is a pretty strong condition on the mean semigroup $(S_t)_{t\geq 0}$. For instance, it excludes the case of super Brownian motions in the whole space. However, it is satisfied in a lot of cases.
	For a long list of (symmetric and non-symmetric) Markov processes satisfying Assumption \ref{asp:2'}, see \cite{RenSongZhang2015Limit}.

		A consequence of this assumption is that (see \cite[Theorem 2.7]{KimSong2008Intrinsic})
there exist constants $c>0$ and $\gamma >0$ such that
\begin{equation}\label{eq:IU}
	\Big|\frac{q(t,x,y)}{\phi(x)\phi^*(y)} - 1\Big| \leq c e^{-\gamma t}, \quad x\in E, t> 1.
\end{equation}
	We will see in Subsection \ref{sec: classcal spine decomposition} that, under Assumption \ref{asp:2}, the spine process $\{(\xi_t)_{t\geq 0}; (\dot{\mathbb P}_x)_{x\in E}\}$ in the classical spine decomposition is a time homogeneous Markov process with invariant measure $\phi(x)\phi^*(x)m(dx)$.
	It can be verified that its transition density with respect to measure $\phi(x)\phi^*(x)m(dx)$ is
$
\frac{q(t,x,y)}{\phi(x)\phi^*(y)}.
$
Therefore Assumption \ref{asp:2'} implies that the spine process in classical spine decomposition is exponentially ergodic.

	Define $\nu(dy) := \phi^*(y)m(dy)$.
	Under Assumption \ref{asp:2'},  $\nu(dy)$ is a finite measure on $E$.
	In fact, according to \eqref{eq:IU}, for $t>0$ large enough, there is a $c'_t>0$ such that $\phi^*(y) \leq q(t,x,y)  (c'_t)^{-1}\phi^{-1}(x)$,
	and clearly, the right hand of this inequality is integrable in $y$ with respect to measure $m$.
    Therefore, we can consider a superprocess $X$ with initial configuration $\nu$.
    Under Assumptions \ref{asp:1} and \ref{asp:2'}, it will be proved in Lemma \ref{lem:discuss-of-assumption2'} that the following statements are equivalent:
\begin{itemize}
\item
	$S_tv_s(x)<\infty$ for some $s>0,t>0$ and some $x\in E$.
\item
	$\mathbf P_\nu (X_t = \mathbf 0) > 0$ for some $t>0$.
\end{itemize}
	Note that, in order to take advantage of \eqref{eq:reason-for-asp2'}, we need $S_t v_s(x)$ to be finite at least for some large $s,t>0$ and  some $x\in E$.
	Therefore, we also need the following assumption:
	
	\begin{asp}{$1'$}\label{asp:1'}
	In addition to Assumption \ref{asp:1}, we further assume that $\mathbf P_\nu (X_t = \mathbf 0) > 0$ for some $t>0$.
\end{asp}
	
	We are now ready to state our Kolmogorov type and Yaglom type limit results for superprocesses:
\begin{thm}\label{thm:Kolmogorov-type-of-theorem}
	Suppose that Assumptions  \ref{asp:1'}, \ref{asp:2'} and \ref{asp:3} hold. Then,
\[
	t\mathbf P_\mu(X_t\neq \mathbf 0)
	\xrightarrow[t\to\infty]{} \frac{\langle \mu,\phi\rangle} {\frac{1}{2}\langle  A \phi,\phi \phi^*\rangle_m},
	\quad \mu\in \mathcal M^\phi_f,
\]
where $m$ is the reference measure appeared in Assumption \ref{asp:2}.
\end{thm}

\begin{thm}\label{thm:Yaglom-type-theorem}
	Suppose that Assumptions  \ref{asp:1'}, \ref{asp:2'} and \ref{asp:3} hold.
	Let $f\in bp\mathscr B^\phi_E$ and $\mu\in\mathcal M^\phi_f$.
	Then,
\[\begin{split}
	\big\{t^{-1}X_t(f);\mathbf P_\mu(\cdot | X_t\neq \mathbf 0)\big\}
	\xrightarrow[t\to\infty]{law} \frac{1}{2}\langle \phi^*, f\rangle_m\langle \phi A, \phi\phi^*\rangle_m \mathbf e,
\end{split}\]
	where $\mathbf e$ is an exponential random variable with mean 1,
 and $m$ is the reference measure  in Assumption \ref{asp:2}.
\end{thm}

	As mentioned earlier, our Kolmogorov type and Yaglom type results for critical superprocesses are established under slightly weaker conditions than \cite{RenSongZhang2015Limit}.
We now make this more precise.
	In \cite{RenSongZhang2015Limit}, the authors considered a $(\xi,\psi)$-superprocess $\{(X_t)_{t\geq 0}; (\mathbf P_\mu)_{\mu \in \mathcal M_f}\}$ which also satisfies Assumption \ref{asp:1}, \ref{asp:2} and \ref{asp:3}.\eqref{asp:3 1} as the basic setting. In addition to that, \cite{RenSongZhang2015Limit} assumed the following
\begin{itemize}
\item[(a)]
	the transition semigroup $(P_t)$ of the spatial motion is intrinsically ultracontractive,
\item[(b)]
	the principal eigenfunction of $(P_t)$ is bounded,
\item[(c)]
	the function $A$ is bounded, and
\item[(d)]
	there exists $t_0>0$ such that $\inf_{x\in E} \mathbf P_{\delta_x}(X_{t_0}= \mathbf 0)>0$.
\end{itemize}
It is shown in \cite{RenSongZhang2015Limit} that, under conditions (a) and (b), the mean semigroup $(S_t)$ is also intrinsically ultracontractive, and the principal eigenfunction $\phi$ of $(S_t)$ is also bounded. Therefore, conditions (a), (b) and (c) combined together are stronger than our Assumption \ref{asp:1'} and \ref{asp:3}. Condition (d) is stronger than our Assumption \ref{asp:2'} because according to \eqref{eq: extinction probability with vt}, we always have the following:
\[
	\mathbf P_\nu(X_t  = \mathbf 0) = \exp\{-\langle v_t,\nu \rangle\} = \exp\{\langle \log \mathbf P_{\delta_\cdot}(X_t = \mathbf 0),\nu \rangle\},\quad t>0.
\]

\section{Size-biased decomposition}
\subsection{Size-biased transform of Poisson random measures}
	In this subsection, we digress briefly from superprocesses and
   prove the size-biased decomposition theorem for Poisson random measures,
	i.e., Theorem \ref{prop:sizBaisPoissRandMeas}.
	Let $(S, \mathscr S)$ be a measurable space with a $\sigma$-finite measure $N$.
Let $\{\mathbf N; P\}$ be a Poisson random measure on $(S, \mathscr S)$ with mean measure $N$.
	Campbell's theorem, see \cite[Proof of Theorem 2.7]{Kyprianou2014Fluctuations} for example,
characterizes the law of $\{ \mathbf N; P \}$ by its Laplace functionals:
\[
	P [e^{-\mathbf N(g)}]
	= e^{-N(1 - e^{-g})},
	\quad g\in p\mathscr S.
\]
	According to \cite[Theorem 2.7]{Kyprianou2014Fluctuations}, we also have that $P [\mathbf N(g)] = N(g)$ for each $g\in \mathscr S$ with $N(|g|) < \infty.$
	By monotonicity, one can verify that
\[
	P [\mathbf N(g)] = N(g),
	\quad g\in p\mathscr S.
\]
\begin{lem}\label{lem:size-biased-lemma}
	If $g\in L^1(N)$ and $ f\in p\mathscr S$, then $\mathbf N(g) e^{-\mathbf N(f)}$ is integrable and
\begin{equation}\label{eq:size-biased-equation}
	P[\mathbf N(g) e^{-\mathbf N(f)}]
	= P[e^{-\mathbf N(f)}] N[g e^{-f}].
\end{equation}
	Furthermore, \eqref{eq:size-biased-equation} is true for each $g,f\in p\mathscr S$ if we allow extended values.	
\end{lem}
\begin{proof}
	Since $N$ is a $\sigma$-finite measure on $(S, \mathscr S)$, there exists a strictly positive measurable function $h$ on $S$ such that $N(h)<\infty$.
	According to \cite[Theorem 2.7.]{Kyprianou2014Fluctuations}, $\mathbf N(h)$ has finite mean.
	For any $g\in bp\mathscr S^h :=\{g\in p\mathscr S: \|h^{-1}g\|_\infty <\infty\}$ and $f\in p\mathscr S$, it is clear that $\mathbf N(g)$ and $\mathbf N(g)e^{-\mathbf N(f)}$ are integrable.
	Therefore, by the dominated convergence theorem, we deduce that
\[\begin{split}
	P[\mathbf N(g) e^{-\mathbf N(f)}]
	&=P[-\partial_\theta|_{\theta=0} e^{-\mathbf N(f+\theta g)}]
	= -\partial_\theta|_{\theta=0} P[e^{-\mathbf N(f+\theta g)}]\\
	&= -\partial_\theta|_{\theta=0} e^{-N(1-e^{-(f+\theta g)})}
	= e^{- N(1-e^{-f})} \partial_\theta|_{\theta=0}N(1-e^{-(f+\theta g)})\\
	&= P[e^{-\mathbf N(f)}]N[ge^{-f}].
\end{split}\]
	For any $g\in p\mathscr S$ and $s\in S$, define $g^{(n)}(s) := h(s) \min\{h(s)^{-1}g(s),n\}$.
	Then $(g^{(n)})_{n\in \mathbb N}$ is a $bp\mathscr S^h$-sequence which increasingly
	converges to $g$ pointwise.
	Note that \eqref{eq:size-biased-equation} is true for each $g^{(n)}$ and $f$.
	Letting $n\to\infty$, by monotonicity, we see that if we allow extended values, then \eqref{eq:size-biased-equation} is true for each $g,f\in p\mathscr S$.
	In the case when $g\in L^1(N)$, we simply consider its positive and negative parts.
\end{proof}

\begin{proof}[Proof of Theorem \ref{prop:sizBaisPoissRandMeas}]
	By Lemma \ref{lem:size-biased-lemma}, it is easy to see that, for any $f\in p\mathscr S$,
\[\begin{split}
	P^{\mathbf N(g)}[e^{-\mathbf N(f)}]
	&= N(g)^{-1} P[\mathbf N(g) e^{-\mathbf N(f)}]
	= N(g)^{-1} P[e^{-\mathbf N(f)}] N[ge^{-f}]\\
	&=  P[e^{-\mathbf N(f)}] N^g[e^{-f}]
	= (P\otimes Q)[e^{-\mathbf N(f) - f(\vartheta)}]
	= (P\otimes Q)[e^{-(\mathbf N + \delta_\vartheta)(f)}],
\end{split}\]
	which completes the proof.
\end{proof}
\begin{lem}\label{lem:covPoissRandMeas}
	For all $g, f \in L^1(N) \cap L^2(N)$, $\mathbf N(g) \mathbf N(f)$ is integrable and
\begin{equation}\label{eq:2rdMomPoissRandMeas}
	P [\mathbf N(g) \mathbf N(f)]
	= N(g) N(f) + N(g f).
\end{equation}
	Furthermore, \eqref{eq:2rdMomPoissRandMeas} is true for all $g,f\in p\mathscr S$ if we allow extended values.
\end{lem}
\begin{proof}
	Since $N$ is a $\sigma$-finite measure on $(S, \mathscr S)$, there exists a strictly positive measurable function $\tilde h$ on $S$ such that $N(\tilde h)<\infty$.
	Define $h(s) := \min\{ \tilde h(s), \tilde h(s)^{1/2} \}$ for each $s \in S$.
	It is clear that $h$ is a strictly positive measurable function on $S$ such that $N(h) < \infty$ and $N(h^2) < \infty$.
	According to \cite[Theorem 2.7]{Kyprianou2014Fluctuations}, $\mathbf N(h)$ has finite 1st and 2nd moments.
	For any $g, f \in bp\mathscr S^h := \{g \in p\mathscr S: \|h^{-1} g\|_\infty < \infty\}$, it is easy to see that $\mathbf N(g), \mathbf N(f), \mathbf N(f) \mathbf N(g)$ are integrable.
	Thus, using Lemma \ref{lem:size-biased-lemma} and the dominated convergence theorem, we have
\[\begin{split}
	P [\mathbf N(g) \mathbf N(f)]
	&= - P[\partial_\theta|_{\theta = 0} \mathbf N(g) e^{-\mathbf N(\theta f)}]
	= - \partial_\theta|_{\theta = 0} P[\mathbf N(g) e^{-\mathbf N(\theta f)}]\\
	&= - \partial_\theta|_{\theta = 0} P[e^{-\mathbf N(\theta f)}]N(ge^{-\theta f})\\
	&= - N[g] \partial_\theta|_{\theta = 0} P[e^{-\mathbf N(\theta f)}] - \partial_\theta|_{\theta = 0} N(ge^{- \theta f}) \\
	&= - N(g) P[ \partial_\theta|_{\theta = 0} e^{-\mathbf N(\theta f)}] - N( \partial_\theta|_{\theta = 0} g e^{-\theta f}) \\
	&= N(g) N(f) + N(g f).
\end{split}\]
	For any $g,f\in p\mathscr S$ and $s \in S$, define $g^{(n)}(s) := h(s) \min\{h(s)^{-1}g(s),n\}$.
	Then $(g^{(n)})_{n\in \mathbb N}$ is a $bp\mathscr S^h$-sequence which increasingly converges to $g$ pointwise.
	Define $f^{(n)}$ similarly.
	Then from what we have proved, \eqref{eq:2rdMomPoissRandMeas} is true for $g^{(n)}$ and $f^{(n)}$.
	Letting $n\to\infty$, by monotonicity, \eqref{eq:2rdMomPoissRandMeas} is true for each $g,f\in p\mathscr S$ if we allow extended values.
	In the case when $g,f\in L^1(N) \cap L^2(N)$ we simply consider their positive and negative parts.
\end{proof}

\subsection{Size-biased transform of the superprocesses.}
\label{sec: size-biased transform of the superprocesses}
	Let $X=\{(X_t)_{t\geq 0}; (\mathbf P_\mu)_{\mu \in \mathcal M_f}\}$ be the $(\xi,\psi)$-superprocess introduced in Subsection \ref{sec: Main results} which satisfies Assumption \ref{asp:1}.
	In this subsection, we will give a proof of Theorem \ref{prop:sizBiasDecSupProc}.
 Recall that, for any $\mu\in \mathcal M_f$, $\{\mathcal N; \mathbf P_\mu\}$ is a Poisson random measure
	with mean measure $\mathbb N_\mu$,
	and our $(\xi,\psi)$-superprocess $(X_t)_{t\ge 0}$ is  given by
\[
	X_0 := \mu;
	\quad X_t(\cdot)
	:=\mathcal N[w_t(\cdot)],
	\quad t>0.
\]

	For any $T>0$, we write $(K,f)\in \mathcal K_T$ if $f: (s,x) \mapsto f_s(x)$ is a bounded non-negative Borel function on $(0,T] \times E$ and $K$ is an atomic measure on $(0,T]$ with finitely many atoms.
	For any $(K,f)\in\mathcal K_T$ and any $\mathcal M_f$-valued  process $(Y_t)_{t> 0}$, we define the random variable
\[	
	K_{(s, T]}^f(Y)
	:= \int_{(s,T]} Y_{r-s}(f_r) K(dr),
	\quad s\in [0,T].
\]	
	It is clear that the two $\mathcal M_f$-valued processes $(Y_t)_{t>0}$ and $(X_t)_{t>0}$ have same finite-dimensional distributions if and only if
\[
	\mathbf E[e^{-K_{(0,T]}^f(X)}]
	=\mathbf E[e^{-K^f_{(0,T]}(Y)}],
	\quad (K,f)\in \mathcal K_T, T>0.
\]

\begin{proof}[Proof of Theorem \ref{prop:sizBiasDecSupProc}]
	Since $\mathbb N_\mu(F) \in (0,\infty)$, it follows from Campbell's formula that $\mathbf P_\mu[\mathcal N(F)] = \mathbb N_\mu(F) \in (0,\infty)$.
	Therefore, $\mathbf P_\mu^{\mathcal N(F)}$ -- the $\mathcal N(F)$-transform of $\mathbb P_\mu$, and $\mathbb N_\mu^F$ --- the $F$-transform of $\mathbb N_\mu$, are both well defined probability measures.
	Notice that, under $\mathbf P_\mu^{\mathcal N(F)}$, $X_0\overset{\text{a.s.}}{=}\mu$ is deterministic,
	and so is $X_0+Y_0$ under  $\mathbf P_\mu \otimes \mathbf Q_\mu$ since $X_0+Y_0\overset{\text{a.s.}}=\mu$.
	Therefore, we only have to show that,
\[
	\{(X_t)_{t > 0}; \mathbf P_{\mu}^{\mathcal N(F)}\}
	\overset{f.d.d.}{=} \{(X_t + Y_t)_{t > 0}; \mathbf P_\mu\otimes \mathbf Q_\mu\}.
\]
	It then immediately follows from Theorem \ref{prop:sizBaisPoissRandMeas} that
\[
	\{\mathcal N; \mathbf P^{\mathcal N(F)}_\mu\}
	\overset{law}{=} \{\mathcal N + \delta_Y; \mathbf P_\mu\otimes \mathbf Q_\mu\}.
\]
	This completes the proof since for any $T > 0$ and $(K,f) \in \mathcal K_T$,
\[\begin{split}
	\mathbf P_\mu^{\mathcal N(F)} [e^{-K_{(0, T]}^f(X)}]
	&= \mathbf P_\mu^{\mathcal N(F)} [e^{-\mathcal N[K_{(0, T]}^f(w)]}]
	= (\mathbf P_\mu\otimes \mathbf Q_\mu) [e^{-(\mathcal N+\delta_Y)[K_{(0, T]}^f(w)]}]\\
	&= (\mathbf P_\mu\otimes \mathbf Q_\mu) [e^{-K_{(0, T]}^f(X+Y)}].
	\qedhere
\end{split}\]
\end{proof}

\section{Spine decomposition of superprocesses}
	The  classical spine decomposition theorem characterizes the superprocess  $X$ after a martingale change of measure, and has been investigated in the literature in different situations, see \cite{EckhoffKyprianouWinkel, EnglanderKyprianou2004Local, LiuRenSong2009LlogL} for example.
	The martingale that is used for the change of measure is defined by $M_t := e^{-\lambda t}X_t(\phi)$, where $\phi$ is the principal eigenfunction of the generator of the mean semigroup of $X$ with $\lambda$ being the corresponding eigenvalue.
	After this martingale change of measure, the transformed process preserves the Markov property, and thus, to prove the spine decomposition theorem, one only needs to focus on the one-dimensional distribution of the transformed process.
\par
	In this section, we generalize this classical result by considering the $X_T(g)$-transform of the superprocess $X$, where $g$ is a non-negative Borel function on $E$.
	If $g$ is not equal to $\phi$, the $X_T(g)$-transformed process is typically not a Markov process.
	So we have to use a different method to develop the theorem.
	Thanks to Theorem \ref{prop:sizBiasDecSupProc},
	we only have to consider the $w_T(g)$-transform of the Kuznetsov measures.
\subsection{Spine decomposition theorem}
	Let $X=\{(X_t)_{t\geq 0}; (\mathbf P_\mu)_{\mu \in \mathcal M_f}\}$ be the $(\xi,\psi)$-superprocess introduced in Subsection \ref{sec: Main results} which satisfies Assumption \ref{asp:1}.
	In this subsection, we will give a proof of Theorem \ref{prop:sizBiasNMeas}.
	Recall that $(\mathbb N_x)_{x\in E}$ are the Kuznetsov measures defined in Lemma \ref{lem: Kuznetsov measures}. We now recall a result from \cite{Li2011Measure-valued} which is useful for calculations related to $(\mathbb N_x)_{x\in E}$.
\begin{lem}[{\cite[Theorems 5.15 and 8.23]{Li2011Measure-valued}}]\label{lem:EquatDescNmeas}
	Under Assumption \ref{asp:1}, for all $T> 0$ and $(K,f) \in \mathcal K_T$, we have
\[
	\mathbb N_\mu \big[ 1 - e^{-K_{(s, T]}^f(w)} \big]
	= \mu(u_s)
	= -\log \mathbf P_{\mu} \big[ e^{-K_{(s, T]}^f(X)} \big],
	\quad s\in [0,T], \mu \in \mathcal M_f,
\]
	where the function $u: (s,x) \mapsto u_s(x)$ on $[0,T] \times E$ is the unique bounded positive solution to the following integral equation:
\[
	u_s(x)
    =\mathbb P_x \Big[\int_{(s,T]} f_r(\xi_{r-s}) K(dr) -
    \int_s^T (\Psi u_r)(\xi_{r-s}) dr \Big],
    \quad s \in [0,T], x \in E.
\]
\end{lem}

We now prove the following lemmas:
\begin{lem}\label{lem:relSpinNMeas}
	For all $x\in E, T>0, (K,f) \in \mathcal K_T$ and $g \in p\mathscr B_E$, we have
\begin{equation}\label{eq:relSpinNMeas}
	\mathbb N_x[w_T(g) e^{-K_{(0, T]}^f(w)}]
	= \mathbb P_x[g(\xi_T) e^{-\int_0^T \psi'(\xi_s,u_s(\xi_s)) ds}],
\end{equation}
	where
\[
	\psi'(x,z)
	:= \partial_z \psi(x,z)
	= - \beta(x) + 2 \alpha(x) z + \int_{(0,\infty)} (1 - e^{-yz}) y \pi(x,dy),
	\quad x \in E, z \geq 0,
\]
	and $u: (s,x) \mapsto u_s(x)$ on $[0,T] \times E$ is defined in Lemma \ref{lem:EquatDescNmeas}.
\end{lem}
\begin{proof}
	We first prove assertion \eqref{eq:relSpinNMeas} in the case when $g \in bp\mathscr B_E$.
	Throughout this proof, we fix $(K,f) \in \mathcal K_T$ and consider $0 \leq \theta \leq 1$.
	Define
\begin{equation}\label{eq:mm-1}
	u_s^\theta(x)
	:=	\mathbb N_x \big[ 1 - e^{ - K_{(s, T]}^f(w) - w_{T-s}(\theta g)} \big],
	\quad s \geq 0, x\in E.
\end{equation}
	Let
\[\begin{split}
	\tilde K(dr)
	&:=	\mathbf 1_{0\leq r<T} K(dr) + \delta_T (dr), \\
	\tilde f_r
	&:=	\mathbf 1_{0\leq r<T} f_r + \mathbf 1_{r=T} \big( K(\{T\}) f_T + \theta g \big).
\end{split}\]
	Then $(\tilde K, \tilde f) \in \mathcal K_T$ and \eqref{eq:mm-1} can be rewritten as
\[
	u_s^\theta(x)
	:=	\mathbb N_x \big[ 1 - e^{ - \tilde K_{(s, T]}^{\tilde f} (w)} \big],
	\quad s \geq 0, x\in E.
\]
	It follows from Lemma \ref{lem:EquatDescNmeas} that, for any
$\theta \geq 0$, $(s,x) \mapsto u^\theta_s(x)$ is the unique bounded positive solution to the equation
\[	
	u_s^\theta(x)
	=\mathbb P_{x} \Big[ \int_{(s,T]} \tilde f_r(\xi_{r-s}) \tilde K(dr) - \int_s^T (\Psi u_r^\theta)(\xi_{r-s}) dr \Big],
	\quad s\in [0,T],x\in E,
\]
	which is equivalent to
\begin{equation}
\label{eq:mewth}
	u_s^\theta(x)
    = \mathbb P_{x} \Big[\int_{(s,T]} f_r(\xi_{r-s}) K(dr) + \theta g(\xi_{T-s}) - \int_s^T (\Psi u^\theta_r)(\xi_{r-s})dr\Big].
\end{equation}
	We claim that $u_s^\theta(x)$ is differentiable in $\theta$ at $\theta = 0$.
	In fact, since
\begin{equation}\label{eq:sbN-1}
	\frac {|e^{-K_{(s, T]}^f(w)-w_{T-s}(\theta g)} - e^{-K_{(s, T]}^f(w)}|} {\theta}
	\leq w_{T-s}(g),
	\quad 0 < \theta \leq 1,
\end{equation}
	and
\begin{equation}
\label{eq:sbN-2}
	\mathbb N_x[w_{T-s}(g)]
	= S_{T-s} g (x)
	= \mathbb P_x[e^{\int_0^{T-s} \beta(\xi_r) dr} g(\xi_{T-s})]
	\leq e^{T \|\beta\|_\infty} \|g\|_\infty,
\end{equation}
	it follows from \eqref{eq:mm-1} and the dominated convergence theorem that
\begin{equation}
\label{eq:mm-3}
	\dot u_s(x)
	:= \partial_\theta|_{\theta=0} u_s^\theta(x)
	= \mathbb N_x[w_{T-s}(g) e^{-K^f_{(s,T]}(w)}]
	\leq e^{T \|\beta\|_\infty} \|g\|_\infty.
\end{equation}
	From \eqref{eq:mm-1}, we also have the following upper bound for $u_s^\theta (x)$ with $0 \leq \theta \leq 1$:
\begin{equation}\label{eq:uppBoundUThet}\begin{split}
	u_s^\theta(x)
	&\leq \mathbb N_x \Big[\int_{(s,T]} w_{r-s}(f_r) K(dr) + w_{T-s}(\theta g)\Big]\\
	&= \int_{(s,T]} \mathbb N_x[w_{r-s}(f_r)] K(dr ) + \mathbb N_x[w_{T-s}(\theta g)]\\
	&\leq e^{T\|\beta\|_\infty} \big( \|f\|_\infty K((0,T])+\|g\|_\infty \big)
	= : L_0.
\end{split}\end{equation}
	By elementary analysis, one can verify that, for each $L>0$, there exists a constant $C_{\psi,L}>0$ such that for each $x\in E$ and $0\leq z, z_0 \leq L$,
\begin{equation}\label{eq:lipschPhi}
	|\psi(x,z_0) - \psi(x,z)|
	\leq C_{\psi,L} |z - z_0|.
\end{equation}
	In fact, one can choose
	$C_{\psi,L} 
	:= \|\beta\|_\infty + 2L\|\alpha\|_\infty + \max\{L,1\} \sup_{x\in E} \int_{(0,\infty)} (y \wedge y^2) \pi(x,dy)$.
	This upper bound also implies that
\[
	|\psi' (x,z)|
	\leq C_{\psi,L},
	\quad x \in E, 0 \leq z \leq L.
\]
	Therefore, we can verify that  
$\mathbb P_x[\int_s^T (\Psi u^\theta_r)(\xi_{r-s}) dr]$ 
	is differentiable in $\theta$ at $\theta = 0$.
	In fact, by \eqref{eq:lipschPhi}, \eqref{eq:uppBoundUThet}, \eqref{eq:mm-1}, \eqref{eq:sbN-1} and \eqref{eq:sbN-2}, we have
\[\begin{split}
	\frac {|(\Psi u_r^\theta)(x)- (\Psi u^0_r)(x)|} {\theta}
	&\leq C_{\psi,L_0} \frac {|u_r^\theta(x) - u_r^0(x)|} {\theta}\\
	&\leq C_{\psi,L_0} \cdot e^{T \|\beta\|_\infty} \|g\|_\infty,
	\quad 0 \leq \theta \leq 1.
\end{split}\]
	Therefore, by the bounded convergence theorem, we have
\begin{equation}
\label{eq:mm-5}
\partial_\theta|_{\theta =0} \mathbb P_x \Big[ \int_s^T (\Psi u_r^\theta)(\xi_{r-s}) dr\Big]
	= \mathbb P_x \Big[ \int_s^T \psi'\big(\xi_{r-s},u_r^0(\xi_{r-s})\big) \dot u_r(\xi_{r-s})~dr \Big].
\end{equation}
	Now, taking $\partial_\theta|_{\theta=0}$ on the both sides of \eqref{eq:mewth}, we obtain from \eqref{eq:mm-5} that
\begin{equation}\label{eq:mm-2}
	\dot u_s (x)
	= \mathbb P_x\Big[ g(\xi_{T-s})- \int_s^T \psi'\big(\xi_{r-s},u_r^0(\xi_{r-s})\big) \dot u_r(\xi_{r-s})~ dr \Big],
	\quad s\in [0,T], x\in E.
\end{equation}
	Notice that the function $\dot u:(s,x) \mapsto \dot u_s(x)$ is bounded on $[0,T] \times E$ by $e^{T \|\beta\|_\infty} \|g\|_\infty$;
	$g$ is bounded on $E$ by $\|g\|_\infty$;
	and $\psi'(x, u_r^0(x))$ is bounded on $E$ by $C_{\psi,L_0}$.
	These bounds allow us to apply the classical Feynman-Kac formula, see \cite[Lemma A.1.5]{Dynkin1993Superprocesses} for example, to equation \eqref{eq:mm-2} and get that
\begin{equation}\label{eq:third}
	\dot u_0(x)
	=  \mathbb P_{x} [g(\xi_T) e^{-\int_0^T \psi'(\xi_s,u_s(\xi_s)) ds}].
\end{equation}
	The desired result when $g \in bp\mathscr B_E$ then follows from \eqref{eq:mm-3} and \eqref{eq:third}.
\par
	In the case when $g \in p\mathscr B_E$, we write $g^{(n)}(x) := \min\{g(x), n\}$ for $x \in E$ and $n \in \mathbb N$.
	Then, from what we have proved, we know that
\[
	\mathbb N_x[w_T(g^{(n)}) e^{- K_{(0,T]}^f(w)}]
	= {\mathbb P_x[g^{(n)}(\xi_T) e^{-\int_0^T \psi'(\xi_s,u_s(\xi_s))ds}]},
	\quad n \in \mathbb N.
\]
	Letting $n \to \infty$ we complete the proof.
\end{proof}
\begin{lem}\label{lem:spinImigrCondSpin}
 Let $T>0, k\in[0, T]$ and $(K,f) \in \mathcal K_T$. Let $\mu\in \mathcal M_f$ and $g\in p\mathscr B_E$ satisfy that $\mu(S_Tg) \in (0,\infty)$. Suppose that $\{(\xi_t)_{0\leq t\leq T}, (Y_t)_{0\leq t\leq T}, \mathbf n_T; \dot {\mathbf P}^{(g,T)}_\mu\}$ is a spine representation of $\mathbb N_\mu^{w_T(g)}$.
	Then, we have
\begin{equation}\label{eq:spinImigrCondSpin}
    -\log \dot{\mathbf P}^{(g,T)}_\mu[e^{-K^f_{(k, T]}(Y)}|\xi]
	=\int_k^T \psi'_0(\xi_{s-k},u_s(\xi_{s-k})) ds,
\end{equation}
	where the function $u$ is defined in Lemma \ref{lem:EquatDescNmeas}.
\end{lem}
\begin{proof}
	Throughout this proof,
we denote by $\mathbf n_{T-k}$ and $\mathbf m^\xi_{T-k}$
	the restriction of $\mathbf n_T$ and $\mathbf m^\xi_T$ on $[0,T-k] \times \mathcal W$ respectively.
	It follows from  properties of Poisson random measures that, conditioned on $\xi$, $\mathbf n_{T-k}$ is a Poisson random measure with mean measure $\mathbf m^\xi_{T-k}$.
\par
	It follows from \eqref{eq:defSpinImmigr} and Fubini's theorem that
\begin{equation}\label{eq:kfy_is_nkf}\begin{split}
    &K_{(k,T]}^f(Y)
    = \int_{(k,T]} Y_{r-k}(f_r) K(dr) \\
	&\quad = \int_{(k,T]} K(dr) \int_{(0,r-k] \times \mathcal M_f} w_{(r-k)-s}(f_r) \mathbf n_T(ds,dw)\\
	&\quad = \int_{(0,T-k] \times \mathcal M_f} \mathbf n_T(ds,dw) \int_{(k+s,T]} w_{r-(k+s)}(f_r) K(dr)\\
    &\quad = \int K^f_{(k + s, T]}(w)\mathbf n_{T-k}(ds,dw).
\end{split}\end{equation}
	Conditioned on $\xi$, it follows from Campbell's formula and Lemma \ref{lem:EquatDescNmeas} that
\[\begin{split}
    &-\log \dot{\mathbf P}^{(g,T)}_\mu[e^{-K^f_{(k, T]}(Y)}|\xi]
    = -\log \dot{\mathbf P}^{(g,T)}_\mu\big[e^{-\int K^f_{(k + s, T]}(w)\mathbf n_{T-k}(ds,dw)}\big|\xi\big]\\
	&\quad = \int(1 - e^{-K_{(k + s, T]}^f (w)})\mathbf m^\xi_{T-k}(ds,dw)\\
	&\quad = \int_0^{T-k} \Big(2\alpha(\xi_s) \mathbb N_{\xi_s}[1 - e^{-K_{(k + s, T]}^f(w)}] \\
	&\qquad\qquad\qquad + \int_{(0,\infty)} y \mathbf P_{y \delta_{\xi_s}}[1 - e^{-K_{(k + s, T]}^f(X)}] \pi(\xi_s,dy)\Big) ds\\
	&\quad = \int_0^{T-k} \Big(2\alpha(\xi_s) u_{k+s}(\xi_s) + \int_{(0,\infty)} (1 - e^{-yu_{k+s}(\xi_s)})y\pi(\xi_s,dy)\Big) ds
	\\&\quad =\int_0^{T-k} \psi'_0\big(\xi_s,u_{s+k}(\xi_s)\big) ds
	=\int_k^T \psi'_0\big(\xi_{s-k},u_s(\xi_{s-k})\big) ds,
\end{split}\]
	as desired.
\end{proof}
\begin{proof}[Proof of Theorem \ref{prop:sizBiasNMeas}]
    We only need to prove that
\[
	\{(Y_t)_{0<t\le T}; \dot{\mathbf P}^{(g,T)}_\mu\}
	\overset{f.d.d.}{=} \{(w_t)_{0<t\le T}; \mathbb N_\mu^{w_T(g)}\},
\]
	since both $\{Y_0; \dot{\mathbf P}^{(g,T)}_\mu\}$ and $\{w_0; \mathbb N_\mu^{w_T(g)}\}$ are deterministic with common value $\mathbf 0$.
	By Lemma \ref{lem:relSpinNMeas} and \ref{lem:spinImigrCondSpin}, we have
	\[\begin{split}
    &\mathbb N_\mu^{w_T(g)}\big[e^{-K_{(0, T]}^f(w)}\big]
	=\mathbb N_\mu[w_T(g)]^{-1} \mathbb N_\mu \big[w_T(g) e^{-K_{(0, T]}^f(w)}\big]\\
	&\quad =\mu(S_Tg)^{-1}  \mathbb P_\mu\big[g(\xi_T) e^{-\int_0^T \psi'(\xi_s,u_s(\xi_s)) ds}\big]\\
	&\quad =\mathbb P^{(g,T)}_{\mu}[e^{-\int_0^T \psi'_0(\xi_s, u_s(\xi_s)) ds}]
	=\dot {\mathbf P}^{(g,T)}_\mu\big[ \dot{\mathbf P}^{(g,T)}_\mu[e^{-K_{(0, T]}^f(Y)}|\xi] \big]\\
	&\quad = \dot{\mathbf P}^{(g,T)}_\mu[e^{-K_{(0, T]}^f(Y)}].
\end{split}\]
	The proof is complete.
\end{proof}

\subsection{Classical spine decomposition theorem}
\label{sec: classcal spine decomposition}
	Let $X=\{(X_t)_{t\geq 0}; (\mathbf P_\mu)_{\mu \in \mathcal M_f}\}$ be the $(\xi,\psi)$-superprocess introduced in Subsection \ref{sec: Main results} which satisfies Assumptions \ref{asp:1} and \ref{asp:2}.
	In this subsection, we will recover the classical spine decomposition theorem for $X$ which is developed previously in \cite{EckhoffKyprianouWinkel, EnglanderKyprianou2004Local, LiuRenSong2009LlogL}.

	It is clear that $\{(e^{-\lambda t}\phi(\xi_t) e^{\int_0^t \beta(\xi_s) ds}\mathbf 1_{t< \zeta})_{t \geq 0}; (\mathbb P_x)_{x\in E}\}$ is a non-negative martingale.
	Denote by $\{(\xi_t)_{t\geq 0}; (\dot{\mathbb P}_x)_{x\in E}\}$ the martingale transform (also known as Doob's $h$-transform) of $\{(\xi_t)_{t\geq 0}; (\mathbb P_x)_{x\in E}\}$ via this martingale in the sense that
\[
	\frac {d \dot{\mathbb P}_x|_{\mathscr F_t^\xi}} {d \mathbb P_x|_{\mathscr F_t^\xi}}
	:=e^{-\lambda t} \frac {\phi(\xi_t)} {\phi(x)} e^{\int_0^t \beta(\xi_s) ds} \mathbf 1_{t< \zeta},
	\quad x \in E,t \geq 0,
\]
	where $(\mathscr F^\xi_t)_{t\geq 0}$ is the natural filtration of the spatial motion $\xi$.
   It can be shown that
	(see \cite{KimSong2008Intrinsic} for example) $\{(\xi_t)_{t\geq 0}; (\dot{\mathbb P}_x)_{x\in E}\}$ is a time homogeneous Markov 
	process.
 Its semigroup is Doob's  $h$-transform of $(S_t)_{t\geq 0}$ with $h=\phi$ and its transition density with respect to the measure $m$ is
\[
	\dot{q}(t,x,y)
	:= e^{-\lambda t}\frac {\phi(y)} {\phi(x)} q(t,x,y),
	\quad x, y\in E, t> 0.
\]
	It can also be verified that $\phi(x)\phi^*(x)m(dx)$ is an invariant measure for $\{(\xi_t)_{t\geq 0}; (\dot{\mathbb P}_x)_{x\in E}\}$.

Recall that, for each $T>0$, $\mathbb P^{(\phi,T)}_\mu$ is defined as the $(e^{\int_0^T\beta(\xi_s)ds} \phi(\xi_T)\mathbf 1_{\zeta < T})$-transform of the measure 
	$\mu\mathbb P(\cdot):= \int_E \mathbb P_x(\cdot)\mu(dx)$.
\begin{lem}
\label{lem: measure of spine}
	Let $\mu \in \mathcal M_f^\phi$. Define
   a probability measure
	$\dot {\mathbb P}_\mu(\cdot):= \mu(\phi)^{-1}\int_E \phi(x)\dot {\mathbb P}_x(\cdot )\mu(dx)$.
Then, for each $T>0$, we have $\{(\xi_t)_{0\leq t\leq T}; \mathbb P_\mu^{(\phi,T)}\}\overset{law}{=}\{(\xi_t)_{0\leq t\leq T}; \dot{\mathbb P}_\mu\}$.
\end{lem}
\begin{proof}
	Let $A\in \mathscr F_T^\xi$. Then we have
\[\begin{split}
	&\mathbb P_\mu^{(\phi,T)}(A) = \frac{(\mu\mathbb P)[\mathbf 1_A e^{\int_0^T\beta(\xi_s)ds} \phi(\xi_T)\mathbf 1_{T < \zeta}]}{(\mu\mathbb P)[e^{\int_0^T\beta(\xi_s)ds} \phi(\xi_T)\mathbf 1_{T < \zeta}]}
	\\&=  \mu(\phi)^{-1}(\mu\mathbb P)[\mathbf 1_A e^{-\lambda T}e^{\int_0^T\beta(\xi_s)ds} \phi(\xi_T)\mathbf 1_{T < \zeta}]
	\\&=  \mu(\phi)^{-1}\int_E\mathbb P_x[\mathbf 1_A e^{-\lambda T}e^{\int_0^T\beta(\xi_s)ds} \phi(\xi_T)\mathbf 1_{T < \zeta}]~\mu(dx)
	\\&=  \mu(\phi)^{-1}\int_E\phi(x)\dot{\mathbb P}_x (A)~\mu(dx) = \dot{\mathbb P}_\mu(A). \qedhere
\end{split}\]
\end{proof}
\par
	Fix a measure $\mu \in \mathcal M^\phi_f$.
	Define $M_t := e^{-\lambda t}X_t(\phi)$ for each $t \geq 0$.
	It is clear that $\{(M_t)_{t\geq 0}; \mathbf P_\mu\}$ is a non-negative martingale.
	Let $\{(X_t)_{t\geq 0}; \mathbf P_\mu^M\}$ be the martingale transform of $\{(X_t)_{t\geq 0}; \mathbf P_\mu\}$ via this martingale in the sense that
\[
	\frac {d \mathbf P_\mu^M |_{\mathscr F_t^X}} {d \mathbf P_\mu |_{\mathscr F_t^X}}
	:= \frac {M_t} {\mu(\phi)},
	\quad t \geq 0.
\]
\par
	We now give the classical spine decomposition theorem:
\begin{thm}[{Spine decomposition, \cite{EckhoffKyprianouWinkel, EnglanderKyprianou2004Local,LiuRenSong2009LlogL}}]
\label{thm:spinDec}
	Suppose that Assumptions \ref{asp:1} and \ref{asp:2} hold.
	Let $\mu \in \mathcal M_f^\phi$.
Let the spine immigration
	$\{(\xi_t)_{t\geq 0}, (Y_t)_{t\geq 0}, \mathbf n; \dot {\mathbf P}_\mu\}$ be defined as follows:
\begin{enumerate}
\item
	The \emph{spine process} $\{(\xi_t)_{t\geq 0}; \dot{\mathbf P}_\mu\}$ is a copy of $\{(\xi_t)_{t\geq 0}; \dot{\mathbb P}_{\mu}\}$.
\item
	The \emph{immigration process} $\{(Y_t)_{t\geq 0}; \dot{\mathbf P}_\mu\}$ is an $\mathcal M_f$-valued process given by
\[
	Y_t
	:= \int_{(0,t] \times \mathcal W} w_{t-s} \mathbf n(ds,dw),
	\quad t \geq 0,
\]
	where, conditioned on $\xi$, $\mathbf n$ is a Poisson random measure on $[0,\infty) \times \mathcal W$ with mean measure
\[
	\mathbf m^\xi(ds,dw)
:=  2 \alpha(\xi_s)\mathbb N_{\xi_s}(dw)\cdot ds
+ \int_{(0,\infty)} y \mathbf P_{y \delta_{\xi_s}}(X\in dw) \pi(\xi_s,dy)\cdot ds.
\]
\end{enumerate}
	Then,
$
    \{(X_t)_{t\geq 0}; \mathbf P_\mu^M\}
    \overset{f.d.d.}{=} \{(X_t + Y_t)_{t\geq 0}; \mathbf P_\mu \otimes\dot{\mathbf P}_\mu\}.
$
\end{thm}	

\begin{proof}
	Fix $T>0$.
	We only need to show that
\[
	\{(X_t)_{t\leq T}; \mathbf P_\mu^M\}
	\overset{f.d.d.}{=} \{(X_t + Y_t)_{t\leq T}; \mathbf P_\mu\otimes\dot{\mathbf P}_\mu\}.
\]
From Lemma \ref{lem: measure of spine},
	we can verify that
\begin{equation}
\label{eq: consisten of dotPphiT}
	\{(Y_t)_{t\leq T}; \dot{\mathbf P}_\mu\}
	\overset{f.d.d.}{=} \{(Y_t)_{t\leq T}; \dot{\mathbf P}^{(\phi,T)}_\mu\}.
\end{equation}
Also it follows easily from the definitions of $\mathbf P_\mu^M$ and $\mathbf P_\mu^{X_T(\phi)}$ that
\begin{equation}
\label{eq: consisten of PXTphi}
	\{(X_t)_{t\leq T}; \mathbf P_\mu^M\}
	\overset{f.d.d.}{=} \{(X_t)_{t\leq T}; \mathbf P_\mu^{X_T(\phi)}\}.	
\end{equation}
	The desired result then follows from Corollary \ref{cro: spine decomposition}.
\end{proof}
\begin{rem}
\label{rem: consistency}
	Lemma \ref{lem: measure of spine} indicates that $\{(\xi_t)_{0\leq t\leq T};\mathbb P_\mu^{(\phi, T)}\}$
		are consistent.
	From \eqref{eq: consisten of PXTphi} we have that $\{(X_s)_{0\leq s\leq T};\mathbf P_\mu^{X_T(\phi)}\}$
		are consistent.
	From \eqref{eq: consisten of dotPphiT} we have that $\{(Y_t)_{t\leq T}; \dot{\mathbf P}^{(\phi,T)}_\mu\}$
		are consistent.
	According to Theorem \ref{prop:sizBiasNMeas}, we have $\{(w_t)_{t\leq T}; \mathbb N^{w_T(\phi)}_\mu\} \overset{f.d.d}{=}\{(Y_t)_{t\leq T}; \dot{\mathbf P}^{(\phi,T)}_\mu\}$ which implies that $\{(w_t)_{t\leq T}; \mathbb N^{w_T(\phi)}_\mu\}$ are
	also consistent.
\end{rem}

\section{2-spine decomposition of critical superprocesses}

\subsection{Second moment formula}
\label{sec: Second moment formula}
Let $X=\{(X_t)_{t\geq 0}; (\mathbf P_\mu)_{\mu \in \mathcal M_f}\}$ be the $(\xi,\psi)$-superprocess introduced in Subsection \ref{sec: Main results} which satisfies Assumptions \ref{asp:1}, \ref{asp:2} and \ref{asp:3}.
In this subsection, we give a second moment formula for  superprocesses.
\begin{lem}
\label{lem:1stMomSizBiasSupProc}
	Suppose that Assumptions  \ref{asp:1}, \ref{asp:2} and \ref{asp:3} hold.
	Let $g,f\in bp\mathscr B_E^\phi, \mu \in \mathcal M^\phi_f$ 
	and $t\geq 0$.
	Suppose that $\{(\xi_s)_{0\leq s\leq t}, (Y_s)_{0\leq s\leq t}, \mathbf n_t; \dot {\mathbf P}^{(g,t)}_\mu\}$ is the spine representation of $\mathbb N_\mu^{w_t(g)}$.
	Then,
\[
	\dot{\mathbf P}_{\mu}^{(g,t)} [Y_t(f)|\xi]
    = \int_0^t A(\xi_s) \cdot (S_{t-s} f) (\xi_s) ds
	\leq t \|A \phi\|_{\infty} \|\phi^{-1}f\|_\infty,
	\quad \dot{\mathbf P}^{(g,t)}_{\mu} \text{-a.s.}.
\]
\end{lem}

\begin{proof}
	Define $G(s,w) := \mathbf 1_{s\leq t} w_{t-s}(f)$ for all $s \geq 0$ and $w \in \mathcal W$.
	Under Assumption \ref{asp:3}, it is clear from \eqref{eq:meanMeasImmigr} that
\[\begin{split}
	\mathbf m_t^\xi(G)
	&= \int_0^t 2 \alpha(\xi_s) \mathbb N_{\xi_s}[w_{t-s}(f)]ds + \int_0^t ds \int_{(0,\infty)} y \mathbf P_{y \delta_{\xi_s}}[X_{t-s}(f)] \pi(\xi_s,dy) \\
	&= \int_0^t 2 \alpha(\xi_s)\cdot (S_{t-s}f)(\xi_s) ds + \int_0^t ds \int_{(0,\infty)} y^2 \cdot (S_{t-s} f)(\xi_s) \pi(\xi_s, dy)\\
    &= \int_0^t A(\xi_s) \cdot (S_{t-s} f) (\xi_s) ds.
\end{split}\]
	Since, conditioned on $\xi$,  $\{\mathbf n_t;\dot{\mathbf P}^{(g,t)}_\mu\}$ is a
	Poisson random measure on $[0,t] \times \mathcal W$
	with mean measure $\mathbf m_t^\xi$, we conclude from Campbell's theorem that
\[
	\dot{\mathbf P}^{(g,t)}_{\mu}[Y_t(f)|\xi]
	= \dot{\mathbf P}^{(g,t)}_{\mu} [\mathbf n_t(G)|\xi]
	= \mathbf m_t^\xi(G)
    = \int_0^t A(\xi_s) \cdot (S_{t-s} f) (\xi_s) ds,
	\quad \dot{\mathbf P}^{(g,t)}_{\mu} \text{-a.s.}.
\]
	Noticing that
\[
   	\int_0^t A(\xi_s) \cdot (S_{t-s} f) (\xi_s) ds
	 = \int_0^t [(A \phi) \phi^{-1} S_{t-s}(\phi \cdot \phi^{-1} f)](\xi_s)ds
	\leq t \|A\phi\|_\infty \|\phi^{-1}f\|_\infty,
\]
	we have our result as desired.
\end{proof}
\begin{prop}\label{prop:covanrance}
	Under Assumptions \ref{asp:1}, \ref{asp:2} and \ref{asp:3}, for all $g,f\in b\mathscr B^\phi_E$, $\mu \in \mathcal M^\phi_f$
	and $t\geq 0$,
	we have that $X_t(g)X_t(f)$ is integrable with respect to $\mathbf P_\mu$ and
\begin{equation}\label{eq:covanrance}
	\mathbf P_\mu[X_t(g) X_t( f)]
	= \langle\mu, S_t g \rangle \langle\mu, S_t f\rangle + \langle\mu, \phi\rangle \dot{\mathbb P}_{\mu} \Big[(\phi^{-1} g)(\xi_t) \int_0^t A(\xi_s) \cdot (S_{t-s} f)(\xi_s) ds\Big].
\end{equation}
\end{prop}
\begin{proof}
	We first consider the case when $g,f\in bp\mathscr B^\phi_E$.
	In this case, the right hand of \eqref{eq:covanrance} is finite.
	Actually,
	by Lemma \ref{lem:1stMomSizBiasSupProc},
	the right side of \eqref{eq:covanrance} is less than or equal to
\[\begin{split}
	&\langle \mu, S_tg \rangle \langle\mu,S_t f \rangle +
\langle \mu,\phi\rangle\dot{\mathbb P}_\mu\big[(\phi^{-1}g)(\xi_t)\big] t \|A\phi\|_\infty \|\phi^{-1}f\|_\infty\\
	&\quad \leq \langle \mu,\phi \rangle^2 + \langle \mu,\phi\rangle t \|A\phi\|_\infty \|\phi^{-1}g\|_\infty\|\phi^{-1}f\|_\infty
	< \infty.
\end{split}\]
	We can also assume that $m(g)>0$.
 Since if $g\in bp\mathscr B_E$ with $m(g)=0$, then according to \eqref{eq: density of S}, \eqref{eq: mean formula} and Lemma \ref{lem: measure of spine}, we have
\begin{align}
	S_t g(x) = \int_{E} q(t,x,y) g(y)m(dy) = 0,
	&\quad t > 0, x\in E,
	\\\mathbf P_\mu [X_t(g)] = \mu(S_t g)  = 0,
	&\quad \mu \in \mathcal M_f, t>0,
	\\\dot{\mathbb P}_\mu [\phi^{-1}g(\xi_t)] = \mathbb P_\mu^{(\phi,t)}[\phi^{-1}g(\xi_t)]
	= \frac{\mu(S_t g)}{\mu(\phi)}
	= 0,
	&\quad \mu \in \mathcal M_f, t>0.
\end{align}
These imply that the both sides of \eqref{eq:covanrance} are $0$.
\par
	Now in the case when $g,f\in bp\mathscr B^\phi_E$ and $m(g)>0$, from
    Theorem \ref{prop:sizBiasNMeas} and
	Lemma \ref{lem:1stMomSizBiasSupProc} we know that, for each $x\in E$,
	\[\begin{split}
	\mathbb N_x^{w_t(g)}[w_t(f)]
	&= \dot{\mathbf P}_{\delta_x}^{(g,t)}[Y_t(f)]
	= \dot{\mathbf P}_{\delta_x}^{(g,t)} \big[\dot{\mathbf P}_{\delta_x}^{(g,t)}[Y_t(f) | \xi]\big]\\
	&= \dot{\mathbf P}_{\delta_x}^{(g,t)} \Big[\int_0^t  A(\xi_s) \cdot (S_{t-s}  f)(\xi_s) ds\Big]
	= \mathbb P_x^{(g,t)} \Big[\int_0^t  A (\xi_s)\cdot (S_{t-s} f)(\xi_s) ds\Big]\\
	&= S_t g(x)^{-1} \mathbb P_{x} \Big[g(\xi_t) e^{\int_0^t \beta(\xi_s) ds} \int_0^t A(\xi_s)\cdot (S_{t-s} f)(\xi_s) ds\Big].
\end{split}\]
	Therefore,
\begin{equation}\begin{split}
	\mathbb N_x[w_t(g) w_t( f)]
	&= \mathbb N_x[w_t(g)] \mathbb N_x^{w_t(g)}[w_t(f)]\\
	&= \mathbb P_x \Big[g(\xi_t) e^{\int_0^t \beta(\xi_s) ds} \int_0^t A(\xi_s)\cdot (S_{t-s} f)(\xi_s) ds\Big]\\
	&= \phi(x) \dot{\mathbb P}_x \Big[(\phi^{-1} g)(\xi_t) \int_0^t A(\xi_s)\cdot (S_{t-s} f)(\xi_s)ds \Big].
\end{split}\end{equation}
	Integrating with $\mu \in \mathcal M_f^{\phi}$, we have
\begin{equation}
\label{eq: sec moment for N measure}
	\mathbb N_\mu[w_t(g)w_t(f)]
	= \langle\mu, \phi\rangle \dot{\mathbb P}_{\mu} \Big[(\phi^{-1} g)(\xi_t) \int_0^t A(\xi_s)\cdot (S_{t-s} f)(\xi_s) ds\Big].
\end{equation}
It then follows from Lemmas \ref{lem: Kuznetsov measures} and
	\ref{lem:covPoissRandMeas} that
\[\begin{split}
	\mathbf P_\mu[ X_t( g) X_t( f)]
	&= \mathbb N_\mu[ w_t( g)] \mathbb N_\mu[ w_t( f)] + \mathbb N_\mu[ w_t( g) w_t( f)]\\
	&= \langle \mu, S_t g\rangle \langle \mu, S_t f\rangle + \langle \mu, \phi\rangle \dot{\mathbb P}_{\mu} \Big[ ( \phi^{-1} g)( \xi_t) \int_0^t (A S_{t-s} f)( \xi_s) ds \Big]
\end{split}\]
	as desired.
	For the more general case when $g,f\in b\mathscr B^\phi_E$, we only need to consider their positive and negative parts.
\end{proof}

\subsection{2-Spine decomposition theorem}
\label{size-biased-equation}
	Let $X=\{(X_t)_{t\geq 0}; (\mathbf P_\mu)_{\mu \in \mathcal M_f}\}$ be the $(\xi,\psi)$-superprocess introduced in Subsection \ref{sec: Main results} which satisfies Assumptions \ref{asp:1}, \ref{asp:2} and \ref{asp:3}.
	In this subsection, we will prove the 2-spine decomposition theorem for  superprocesses, i.e., Theorem \ref{prop:2-spine-decomposition}.
	
	First, we give a lemma which says that $\mathbb N_\mu^{w_T(\phi)^2}$ --- the $w_T(\phi)^2$-transform of $\mathbb N_\mu$, and $\ddot{\mathbb P}^{(T)}_\mu$ --- the $(\int_0^T (A\phi)(\xi_s) ds)$-transform of $\dot{\mathbb P}_{\mu}$, are both well defined probability measures.
\begin{lem}
	$\mathbb N_\mu[w_T(\phi)^2]
	= \mu(\phi) \dot{\mathbb P}_{\mu} [ \int_0^T (A\phi)(\xi_s) ds]\in (0,\infty)$ for all $\mu \in \mathcal M_f^\phi$ and $T>0$.
\end{lem}
\begin{proof}
	According to \eqref{eq: sec moment for N measure}, we have
\begin{equation}
	\mathbb N_\mu[w_T(\phi)^2 ]
	= \mu(\phi) \dot{\mathbb P}_\mu \Big[ \int_0^T( A \phi)(\xi_s)ds \Big] \leq \mu(\phi) T \|A\phi\|_\infty <  \infty.
\end{equation}
	According to $\mathbb N_\mu[w_T(\phi)] = \mu(\phi) > 0$, we must have $\mathbb N_\mu[w_T(\phi)^2 ] > 0$.
\end{proof}
\begin{rem}\label{rem:chain-rule}
	Note that $\mathbb N^{w_T(\phi)^2}_\mu$ is also the $w_T(\phi)$-transform of $\mathbb N^{w_T(\phi)}_\mu$.
	In fact, the size-biased transforms satisfy the following chain rule:
	If $g,f$ are non-negative measurable functions on some measure space $(D,\mathscr F_D,\mathbf D)$ with $\mathbf D(g) \in (0,\infty)$ and $\mathbf D(gf) \in (0,\infty)$.
	Denoted by $\mathbf D^g$ the $g$-transform of $\mathbf D$, then $(\mathbf D^g)^f = \mathbf D^{gf}$, i.e., the $f$-transform of $\mathbf D^g$ is the $gf$-transform of $\mathbf D$.
	This is true because it is easy to see that
\[
	\mathbf D^{gf}(ds)
	:= \frac{g(s) f(s) \mathbf D(ds)}{\mathbf D[gf]}
	= \frac{f(s) \mathbf D^g(ds)}{\mathbf D^g[f]}
    = (\mathbf D^g)^f (ds),
	\quad s\in S.
\]
\end{rem}

 For each $\mu \in \mathcal M_f^\phi$, let the spine immigration $\{(\xi_t)_{t\geq 0}, (Y_t)_{t\geq 0}, \mathbf n; \dot {\mathbf P}_\mu\}$ be given by Theorem \ref{thm:spinDec}.
	We first state a property of $\{Y; \dot{\mathbf P}_\mu\}$, which is needed later.
\begin{lem}\label{lem:Y-is-immortal}
	$\dot{\mathbf P}_\mu(Y_t = \mathbf 0) = 0$ for all $\mu \in \mathcal M_f^\phi$ and $t>0$.
\end{lem}
\begin{proof}
	According to Theorem \ref{prop:sizBiasNMeas}, we have
\[
	\dot{\mathbf P}_\mu (Y_t = 0)
	= \mathbb N^{w_t(\phi)}_\mu ( w_t(\phi) = 0 )
	= \langle \mu, \phi \rangle^{-1} \mathbb N_\mu [ w_t(\phi) \mathbf 1_{w_t(\phi) = 0}]
	= 0. \qedhere
\]
\end{proof}
\par
	The proof of Theorem \ref{prop:2-spine-decomposition} relies on the following lemma:
\begin{lem}
\label{lem:key-lemma}
	For any $\mu\in \mathcal M_f^\phi$, $T>0$ and $(K,f)\in \mathcal K_T$, we have
\[\begin{split}
    &\dot{\mathbf P}_\mu [Y_T(\phi) e^{-K_{(0, T]}^f(Y)}|\xi]\\
    &\quad =\dot{\mathbf P}_\mu[e^{-K_{(0, T]}^f(Y)}|\xi]\int_0^T  (A\phi)(\xi_s)\dot{\mathbf P}_{\delta_{\xi_s}}[e^{-K_{(s, T]}^f(Y)}]\widetilde {\mathbf P}_{\xi_s}[e^{-K^f_{(s, T]}(X)}]ds,
\end{split}\]
where $\widetilde {\mathbf P}_{x}$ is defined by \eqref{eq:def-tilde-P} for each $x\in E$.
\end{lem}
\begin{proof}
	Define $G(s,w) := \mathbf 1_{s\leq T}w_{T-s}(\phi)$ for all $s\geq 0$ and $w\in\mathcal W$.
	Notice that from \eqref{eq:kfy_is_nkf}, under the probability $\dot {\mathbf P}_\mu$,
	we have $Y_T(\phi) = \mathbf n(G)$ and
	$K_{(0, T]}^f(Y) = \mathbf n(K^f_{(s,T]}(w))$.
	From Lemmas \ref{lem:1stMomSizBiasSupProc} and \ref{lem:Y-is-immortal} we know that
\[
	0
	< \dot{\mathbf P}_{\mu}[Y_T(\phi)|\xi]
	< \infty,
	\quad \dot{\mathbf P}_\mu \text{-a.s.}.
\]
	Therefore, we can apply Lemma \ref{lem:size-biased-lemma} to the conditioned
	Poisson random measure $\mathbf n$, and get
\begin{equation}\label{eq:condSizBiasEquatSpinImmigr}
    \dot{\mathbf P}_\mu [\mathbf n(G) e^{-\mathbf n(K^f_{(s, T]}(w))}|\xi]
	=\dot{\mathbf P}_\mu[e^{-\mathbf n(K^f_{(s, T]}(w))}|\xi]\mathbf m^\xi[Ge^{-K^f_{(s, T]}(w)}].
\end{equation}
	It is clear from the definitions of $\mathbf m^\xi$, $\mathbb N^{w_t(\phi)}$ and $\mathbf P^M$ that
\begin{equation}\label{eq:represent-mGeKf}\begin{split}
    \mathbf m^\xi[Ge^{-K_{(s, T]}^f(w)}]
	&=\int_0^T \Big( 2\alpha(\xi_s)\mathbb N_{\xi_s}[w_{T-s}(\phi)e^{-K^f_{(s,T]}(w)}]\\
	&\qquad\qquad +\int_{(0,\infty)}y\mathbf P_{y\delta_{\xi_s}}[X_{T-s}(\phi)e^{-K^f_{(s,T]}(X)}]\pi(\xi_s,dy)\Big)ds\\
	&=\int_0^T \Big(2(\alpha\phi)(\xi_s)\mathbb N^{w_{T-s}(\phi)}_{\xi_s}[e^{-K^f_{(s,T]}(w)}] \\
	&\qquad\qquad + \int_{(0,\infty)}y^2\phi(\xi_s)\mathbf P^M_{y\delta_{\xi_s}}[e^{-K^f_{(s,T]}(X)}]\pi(\xi_s,dy)\Big)ds.
\end{split}\end{equation}
	According to  Theorem \ref{prop:sizBiasNMeas}, we have
\begin{equation}
\label{eq:represent-NxwTsphi}
	\mathbb N_x^{w_{T-s}(\phi)}[e^{-K^f_{(s,T]}(w)}]
	= \dot{\mathbf P}_{\delta_x}[e^{-K^f_{(s,T]}(Y)}]
	= \dot{\mathbf P}_{\delta_x}[e^{-K^f_{(s,T]}(Y)}]\mathbf P_{\mathbf 0}[e^{-K^f_{(s,T]}(X)}],
\end{equation}
	where we used the fact that $\mathbf P_{\mathbf 0}(X_t=\mathbf 0,\mbox{ for any }t\ge 0)=1$.
It follows from Theorem \ref{thm:spinDec} that for any $s\in [0,T], x\in E$ and $y\in (0,\infty)$,
\begin{equation}\label{eq:represent-PMydeltax}
	\mathbf P^M_{y\delta_x}[e^{-K^f_{(s,T]}(X)}]
	= \dot{\mathbf P}_{y\delta_x}[e^{-K^f_{(s,T]}(X+Y)}]
	=\dot{\mathbf P}_{\delta_x}[e^{-K^f_{(s,T]}(Y)}]\mathbf P_{y\delta_x}
	[e^{-K^f_{(s,T]}(X)}].
\end{equation}
	Plugging \eqref{eq:represent-NxwTsphi} and \eqref{eq:represent-PMydeltax} back into \eqref{eq:represent-mGeKf} and rearranging terms, we have that
\begin{equation}\label{eq:fifth}
\begin{split}
    &\mathbf m^\xi[Ge^{-K_{(s, T]}^f}(w)]\\
	&\quad=\int_0^T \Big(2(\alpha\phi)(\xi_s)\dot{\mathbf P}_{\delta_{\xi_s}}[e^{-K^f_{(s,T]}(Y)}]\mathbf P_{\mathbf 0}[e^{-K^f_{(s,T]}(X)}] \\
	&\qquad\qquad\quad + \int_{(0,\infty)}y^2\phi(\xi_s)\dot{\mathbf P}_{\delta_{\xi_s}}[e^{-K_s^f(Y)}]\mathbf P_{y\delta_{\xi_s}}[e^{-K^f_{(s,T]}(X)}]\pi(\xi_s,dy)\Big)ds.\\
	&\quad=\int_0^T \phi(\xi_s)\dot{\mathbf P}_{\delta_{\xi_s}}[e^{-K^f_{(s,T]}(Y)}] \\
	&\qquad\qquad\times\Big(2\alpha(\xi_s)\mathbf P_{\mathbf 0}[e^{-K^f_{(s,T]}(X)}] + \int_{(0,\infty)}y^2\mathbf P_{y\delta_{\xi_s}}[e^{-K^f_{(s,T]}(X)}]\pi(\xi_s,dy)\Big)ds\\
	&\quad =\int_0^T (A\phi)(\xi_s)\dot{\mathbf P}_{\delta_{\xi_s}}[e^{-K^f_{(s,T]}(Y)}]\widetilde {\mathbf P}_{\xi_s}[e^{-K^f_{(s,T]}(X)}]ds.
\end{split}\end{equation}
	Plugging \eqref{eq:fifth} back into \eqref{eq:condSizBiasEquatSpinImmigr}, we get the desired result.
\end{proof}
\begin{proof}[Proof of Theorem \ref{prop:2-spine-decomposition}]
 Note that $\{Z_0;\ddot{\mathbf P}^{(T)}_\mu\}$ and $\{w_0;\mathbb N^{w_T(\phi)^2}_\mu\}$ are both deterministic with common value $\mathbf 0$.
	So we only have to proof
$
	\{(Z_t)_{0< t\leq T};\ddot{\mathbf P}^{(T)}_\mu\} \overset{f.d.d.}{=} \{(w_t)_{0<t\leq T};\mathbb N^{w_T(\phi)^2}_\mu\}.	
$
	In order to show this, according to Theorem \ref{prop:sizBiasNMeas} and Remark \ref{rem:chain-rule}, we only need to show that
$\{(Z_t)_{0< t\leq T};\ddot{\mathbf P}^{(T)}_\mu\}$ is the $Y_T(\phi)$-transform of process $\{(Y_t)_{0< t\leq T}; \dot {\mathbf P}_\mu\}$.
\par
 Let $(K,f)\in \mathcal K_T$. 
 Similar to \eqref{eq:kfy_is_nkf}, 
 we have $K_{(r, T]}^f(Y) = \mathbf n_T[K^f_{(r+\cdot, T]}]$ and $K_{(r, T]}^f(Y') = \mathbf n'_T[K^f_{(r+\cdot, T]}]$ for each $r\leq T$.
	Therefore, 
using Campbell's theorem and an argument similar to that used in the proof of Lemma \ref{lem:spinImigrCondSpin}, 
	one can verify that
\begin{equation}\label{eq:mainImmigrCondSkel}
	- \log\ddot {\mathbf P}_\mu [e^{-K^f_{(0, T]}(Y)}|\mathscr G]
	=
	 \int_0^T\psi'_0\big(\xi_s,u_s(\xi_s)\big)ds
\end{equation}
	and
\begin{equation}\label{eq:auxilImmigrCondSkel}
	- \log\ddot {\mathbf P}_\mu [e^{-K^f_{(0, T]}(Y')}|\mathscr G]
	=
 \int_\kappa^T \psi'_0\big(\xi'_s,u_s(\xi'_s)\big)ds,
\end{equation}
	where $u:(s,x)\mapsto u_s(x)$ is the function on $[0,T]\times E$ defined in Lemma \ref{lem:EquatDescNmeas}.
	It is then clear from \eqref{eq:auxilImmigrCondSkel}, \eqref{eq:defAuxilSpin} and Lemma \ref{lem:spinImigrCondSpin} that
\begin{equation}
\label{eq:auxilImmigrCondMainSpin}
\begin{split}
    \ddot{\mathbf P}_\mu[e^{-K^f_{(0, T]}(Y')}|\xi,\kappa]
	&= \ddot{\mathbf P}_\mu[e^{-\int_\kappa^T\psi'_0(\xi'_s,u_s(\xi'_s))ds}|\xi,\kappa]\\
	&= \dot{\mathbb P}_{\xi_r}[e^{-\int_r^T\psi'_0(\xi_{s-r},u_s(\xi_{s-r}))ds}]|_{r=\kappa}
	= \dot{\mathbf P}_{\delta_{\xi_r}} [e^{-K^f_{(r, T]}(Y)}]|_{r=\kappa}.
\end{split}\end{equation}
	By the construction of the splitting immigration $X'$ at time $\kappa$, we also have
\begin{equation}\label{eq:oneTimImmigrCondSkel}
    \ddot {\mathbf P}_\mu[e^{-K^f_{(0, T]}(X')}|\mathscr G ]
	= \widetilde{\mathbf P}_{\xi_r}[e^{-K^f_{(r, T]}(X)}]|_{r=\kappa}.
\end{equation}
	Using \eqref{eq:mainImmigrCondSkel}, \eqref{eq:auxilImmigrCondMainSpin}, \eqref{eq:oneTimImmigrCondSkel} and the construction of the 2-spine immigration, we deduce that
\[\begin{split}
	&\ddot {\mathbf P}_\mu[e^{-K^f_{(0, T]}(Z)}|\xi,\kappa]
	=\ddot {\mathbf P}_\mu\big[ \ddot {\mathbf P}_\mu[e^{-K^f_{(0, T]}(Z)}|\mathscr G ] \big | \xi,\kappa\big]\\
	&\quad=\ddot {\mathbf P}_\mu\Big[ \ddot {\mathbf P}_\mu[e^{-K^f_{(0, T]}(Y)}|\mathscr G ] \ddot {\mathbf P}_\mu[e^{-K^f_{(0, T]}(Y')}|\mathscr G ] \ddot {\mathbf P}_\mu[e^{-K^f_{(0, T]}(X')}|\mathscr G ] \Big | \xi,\kappa\Big]\\
	&\quad=e^{-\int_0^T\psi'_0(\xi_s,u_s(\xi_s))ds}\dot {\mathbf P}_{\delta_{\xi_r}}[ e^{-K^f_{(r, T]}(Y)}]\widetilde{\mathbf P}_{\xi_r}[e^{-K^f_{(r, T]}(X)}]\big |_{r=\kappa}.
\end{split}\]
	Therefore, from the conditioned law of $\kappa$ given $\xi$, we have
\begin{equation}\label{eq:totImmigrCondSpin}\begin{split}
	&\ddot {\mathbf P}_\mu[e^{-K^f_{(0, T]}(Z)}|\xi]\\
	&\quad = \frac{e^{-\int_0^T\psi'_0(\xi_s,u_s(\xi_s))ds}}{\int_0^T (A\phi)(\xi_r)dr}\int_0^T  (A\phi)(\xi_r)\dot {\mathbf P}_{\delta_{\xi_r}}[ e^{-K^f_{(r, T]}(Y)}]\widetilde{\mathbf P}_{\xi_r}[e^{-K^f_{(r, T]}(X)}] dr.
\end{split}\end{equation}
	Taking expectation, we get that
\[\begin{split}
	&\ddot {\mathbf P}_\mu[e^{-K^f_{(0,T]}(Z)}]\\
	&\quad \overset{\text{\eqref{eq:totImmigrCondSpin}}}{=}\ddot {\mathbb P}_{\mu}^{(T)}\Big\{ \frac{e^{-\int_0^T\psi'_0(\xi_s,u_s(\xi_s))ds}}{\int_0^T (A\phi)(\xi_r)dr}\int_0^T  (A\phi)(\xi_r)\dot {\mathbf P}_{\delta_{\xi_r}}[ e^{-K^f_{(r, T]}(Y)}]\widetilde{\mathbf P}_{\xi_r}[e^{-K^f_{(r, T]}(X)}] dr\Big\}\\
	&\quad= \dot{\mathbb P}_{\mu}
	\Big\{ \frac{e^{-\int_0^T\psi'_0(\xi_s,u_s(\xi_s))ds}}{\dot{\mathbb P}_{\mu}[\int_0^T (A\phi)(\xi_r)dr]}\int_0^T  (A\phi)(\xi_r)\dot {\mathbf P}_{\delta_{\xi_r}}[ e^{K^f_{(r, T]}(Y)}]\widetilde{\mathbf P}_{\xi_r}[e^{-K^f_{(r, T]}(X)}] dr\Big\}\\
	&\quad \overset{\text{\eqref{eq:spinImigrCondSpin}}}{=} \dot{\mathbf P}_{\mu}\Big\{ \frac{\dot {\mathbf P}_\mu[e^{-K^f_{(0,T]}(Y)}|\xi]}{\dot{\mathbf P}_{\mu}[\int_0^T (A\phi)(\xi_r)dr]}\int_0^T  (A\phi)(\xi_r)\dot {\mathbf P}_{\delta_{\xi_r}}[ e^{-K^f_{(r, T]}(Y)}]\widetilde{\mathbf P}_{\xi_r}[e^{-K^f_{(r, T]}(X)}] dr\Big\}\\
	&\quad \overset{\text{Lemma \ref{lem:key-lemma}}}{=} \dot{\mathbf P}_{\mu}\Big\{ \frac{\dot {\mathbf P}_\mu[Y_T(\phi)e^{-K^f_{(0, T]}(Y)}|\xi]}{\dot{\mathbf P}_{\mu}[Y_T(\phi)]}\Big\}=\frac{\dot {\mathbf P}_\mu[Y_T(\phi)e^{-K^f_{(0, T]}(Y)}]}{\dot{\mathbf P}_{\mu}[Y_T(\phi)]},
\end{split}\]	
	where in the second equality we used the definition of $\ddot {\mathbb P}_{\mu}^{(T)}$.
	The display above says that $(Z_t)_{0< t\leq T}$ is the $Y_T(\phi)$-transform of the process $\{(Y_t)_{0< t\leq T}; \dot {\mathbf P}_\mu\}$, as desired.
\end{proof}

\section{The asymptotic behavior of critical superprocesses.}\label{sec:asymptotic}

\subsection{Intrinsic ultracontractivity}
\label{sec:further_assumptions}
  Let $\{(X_t)_{t\geq 0}; (\mathbf P_\mu)_{\mu \in \mathcal M_f}\}$ be the $(\xi,\psi)$-superprocess introduced in Subsection \ref{sec: Main results} which satisfies Assumptions \ref{asp:1} and \ref{asp:2'}. 
  In this subsection, we give some more results related to intrinsic ultracontractivity.

\begin{lem}
\label{lem:ergodic}
	Suppose that $F(x,u,t)$ is a bounded Borel function on $E\times [0,1]\times [0,\infty)$ such that $F(x,u):= \lim_{t\to\infty} F(x,u,t)$ exists for all $x\in E$ and $u\in [0,1]$.
	Then we have,
\[
	\int_0^1 F(\xi_{ut},u,t)du
	\xrightarrow[t\to\infty]{L^2(\dot{\mathbb P}_x)} \int_0^1 \langle F(\cdot ,u),\phi\phi^* \rangle_m du,
	\quad x\in E.
\]
\end{lem}
\begin{proof}
\par
	We first show that
\begin{equation}\label{eq:ergodic-lemma-1}
	\dot{\mathbb P}_x[F(\xi_{ut},u,t)]
	\xrightarrow[t\to\infty]{} \langle F(\cdot,u),\phi\phi^*\rangle_m,
      \quad x\in E, u\in(0,1).
\end{equation}
	In fact,
\[
	\dot{\mathbb P}_x[F(\xi_{ut},u,t)]
	=\int_E \frac{\dot{q}(ut,x,y)}{(\phi\phi^*)(y)}F(y,u,t)(\phi\phi^*)(y)m(dy).
\]
	Note that $\int_\cdot (\phi\phi^*)(y)m(dy)$ is a finite measure, $(y,t)\mapsto \frac{\dot{q}(ut,x,y)}{(\phi\phi^*)(y)}F(y,u,t)$ is bounded by $(1+ce^{-\gamma ut})\|F\|_\infty$ for $t>u^{-1}$, and $\frac{\dot{q}(ut,x,y)}{(\phi\phi^*)(y)}F(y,u,t)\xrightarrow[t\to\infty]{}F(y,u)$.
	Using the bounded convergence theorem, we get \eqref{eq:ergodic-lemma-1}.
	By Fubini's theorem,
\[
	\dot{\mathbb P}_x\big[\int_0^1F(\xi_{ut},u,t)du\big]
	=\int_0^1\dot{\mathbb P}_x[F(\xi_{ut},u,t)]du,
	\quad x\in E.
\]
	Since $\dot{\mathbb P}_x[F(\xi_{ut},u,t)]$ is bounded by $\|F\|_\infty$ and $\dot{\mathbb P}_x[F(\xi_{ut},u,t)]\xrightarrow[t\to\infty]{}\langle F(\cdot, u),\phi\phi^*\rangle_m$, by the bounded convergence theorem, we get
\[
	\dot{\mathbb P}_x\big[\int_0^1F(\xi_{ut},u,t)du\big]
	\xrightarrow[t\to\infty]{} c_F
	:=\int_0^1\langle F(\cdot,u),\phi\phi^*\rangle_mdu.
\]
Using \eqref{eq:IU} and a similar argument,
	one can verify that for any $0< u< v\leq 1$,
\[\begin{split}
	&\dot{\mathbb P}_x[F(\xi_{ut},u,t)F(\xi_{vt},v,t)]\\
	&\quad=\int_E\int_E  \dot{q}(ut,x,y) \dot{q}((v-u)t,y,z)F(y,u,t)F(z,v,t)m(dy)m(dz)\\
	&\quad\xrightarrow[t\to\infty]{} \langle F(\cdot,u),\phi\phi^*\rangle_m \langle F(\cdot,v),\phi\phi^*\rangle_m.
\end{split}\]
	The above convergence is also true for $0< v < u\leq 1$ since the limit is symmetric in $u$ and $v$.
	We have again, by Fubini's theorem and the bounded convergence theorem,
\[\begin{split}
	\dot{\mathbb P}_x\Big[\big(\int_0^1 F(\xi_{ut},u,t)du\big)^2\Big]
	&=\int_0^1du\int_0^1\dot{\mathbb P}_x[F(\xi_{ut},u,t)F(\xi_{vt},v,t)]dv\xrightarrow[t\to\infty]{} c_F^2.
\end{split}\]
	Finally, we have
\[\begin{split}
	&\dot{\mathbb P}_x\Big[\big(\int_0^1 F(\xi_{ut},u,t)du - c_F\big)^2\Big]\\
	&\quad =\dot{\mathbb P}_x\Big[\big(\int_0^1 F(\xi_{ut},u,t)du\big)^2\Big] -2c_F \dot{\mathbb P}_x\Big[\int_0^1 F(\xi_{ut},u,t)du\Big]+ c_F^2\\
	&\quad \xrightarrow[t\to\infty]{} 0,
\end{split}\]
	as desired.
\end{proof}

	As mentioned earlier in Subsection \ref{sec: Main results}, in order to study the asymptotic behavior of $(v_t)_{t\geq 0}$ and take advantage of \eqref{eq:reason-for-asp2'}, we need $S_t v_s(x)$ to be finite  at least for some large $s,t>0$ and for some $x\in E$.
	The following lemma addresses this need.

\begin{lem}\label{lem:discuss-of-assumption2'}
 Under Assumption \ref{asp:1} and \ref{asp:2'}, the following statements are equivalent:
\begin{enumerate}
\item[$(1)$]
	$S_tv_s(x)<\infty$ for some $s>0,t>0$ and $x\in E$.
\item[$(1')$]
	There is an $s_0>0$ such that for any $s\geq s_0$, $t>0$ and $x\in E$, we have $S_tv_s(x)<\infty$.
\item[$(2)$]
	$ \langle v_s, \phi^* \rangle_m < \infty$ for some $s>0$.
\item[$(2')$]
	There is an $s_0>0$ such that for any $s\geq s_0$, we have $\langle v_s, \phi^* \rangle_m < \infty$.
\item[$(3)$]
	There is an $s_0>0$ such that for any $s\geq s_0$, we have $v_s\in bp\mathscr B^\phi_E$.
\item[$(4)$]
	$\mathbf P_\nu (X_t = \mathbf 0) > 0$ for some $t>0$.
\item[$(5)$]
	$\phi^{-1}v_t$ converges to $0$ uniformly when $t\to\infty$.
\item[$(6)$]
	For any $\mu \in \mathcal M^\phi_f$, $\mathbf P_\mu(\exists t>0,~s.t.~X_t = \mathbf 0)
	= 1$.
\end{enumerate}
\end{lem}
\begin{proof}
We first give some estimates. In this proof, we allow the extended value $+\infty$.
	According to \eqref{eq: extinction probability with vt} and the fact that $\mathbf 0$ is an absorption state of the superprocess $X$, we have
\begin{equation}\label{eq:assertion1}
\begin{split}
	\langle v_{s_0},\phi^* \rangle_m
	&= -\log \mathbf P_\nu ( X_{s_0} = \mathbf 0 )\\
	& \geq -\log \mathbf P_\nu ( X_s = \mathbf 0 )
	= \langle v_s,\phi^* \rangle_m,
	\quad 0 < s_0\leq s.
\end{split}\end{equation}
	According to Assumption \ref{asp:2'}, we have for each $t\geq 0$, there is a $c_t>0$ such that $q(t,x,y)\leq c_t \phi(x)\phi^*(y)$. 
 Using an argument similar to that of \cite[Proposition 2.5]{KimSong2008Intrinsic}, 
	we have for each $t\geq 0$, there is a $c'_t < 0$ such that $q(t,x,y) \geq c'_t \phi(x)\phi^*(y)$.
	Therefore, we have
\begin{equation}\label{eq:assertion2}\begin{split}
	\phi(x) \langle v_s , \phi^* \rangle_m c'_t
\leq S_t v_s(x) \leq \phi(x) \langle v_s , \phi^* \rangle_m c_t,
	\quad s>0,t>0,x\in E.
\end{split}\end{equation}
	Let $c,\gamma > 0$ be the constants in \eqref{eq:IU}. 
	Notice that $\phi$ is strictly positive, using \eqref{eq:mean-fkpp}, one can verify that
\begin{equation}\label{eq:inequality-with-IU-assumption}\begin{split}
	\frac{V_t f(x)}{\phi(x)}
	\leq \frac{S_t f(x)}{\phi(x)}
	\leq (1+ ce^{-\gamma t}) \langle f,\phi^*\rangle,
	\quad f\in bp\mathscr B_E,x\in E,t>1.
\end{split}\end{equation}
	Taking $f = V_s(\theta \mathbf 1_E)$ in \eqref{eq:inequality-with-IU-assumption} and letting $\theta \to \infty$,
	by \eqref{eq: defi of vt} and \eqref{eq: simigroup for small vt},
	we have that,
\begin{equation}\label{eq:assertion3}
	\frac{v_{t+s}(x)}{\phi(x)}
	\leq (1+ ce^{-\gamma t}) \langle v_s,\phi^*\rangle_m,
	\quad x\in E,s>0,t>1.
\end{equation}
	We can also verify that
\begin{equation}\label{eq:assertion4}
	S_t v_s(x)
	\leq \|\phi^{-1}v_s\|_\infty  S_t \phi(x)
	= \|\phi^{-1}v_s\|_\infty \phi(x)
	\quad s,t>0,x\in E.
\end{equation}
\par
	Now, we are ready to give the proof of this lemma using the following steps: 
	$(1')\Rightarrow (1)\Rightarrow (2)\Rightarrow (2')\Rightarrow (3)\Rightarrow (1')$
	and $(2)\Rightarrow (5)\Rightarrow (6)\Rightarrow (4)\Rightarrow (2)$.
	In fact, it is obvious that 
	$(1')\Rightarrow (1)$.
	For $(1)\Rightarrow (2)$ we use \eqref{eq:assertion2}.
	For $(2)\Rightarrow (2')$ we use \eqref{eq:assertion1}.
	For $(2')\Rightarrow (3)$ we use \eqref{eq:assertion3}.
	For $(3)\Rightarrow (1')$ we use \eqref{eq:assertion4}.
\par
	For $(2)\Rightarrow (5)$, we follow the argument in \cite[Lemma 3.3]{RenSongZhang2015Limit}.
	Note that, from what we have proved, $(2)$ is equivalent to 
	$(1),(1'),(2')$ and $(3)$.
Integrating \eqref{eq:mean-fkpp} with respect to the measure $\nu$,
	by Fubini's theorem and monotonicity, we have that, for any $f\in p\mathscr B_E$ and $t\geq 0$,
\begin{equation}\label{eq:FKPP_mild}\begin{split}
	\langle f,\phi^*\rangle_m
	&=\langle f,S_t^*\phi^*\rangle_m
	=\langle S_tf,\phi^*\rangle_m\\
	&=\langle V_tf,\phi^*\rangle_m + \int_0^t\langle S_{t-r}\Psi_0V_rf,\phi^*\rangle_m dr\\
	&=\langle V_tf,\phi^*\rangle_m + \int_0^t \langle \Psi_0V_rf,\phi^*\rangle_mdr.
\end{split}\end{equation}
	Define
\[
 	v(x)
 	:= \lim_{t\to\infty} v_t(x)
 	= \lim_{t\to\infty}(-\log\mathbf P_{\delta_x}(X_t=\mathbf 0))
 	= -\log \mathbf P_{\delta_x}(\exists t>0, ~\text{s.t.}~ X_t=\mathbf 0).
 \]
	Since $v_t(x)=-\log\mathbf P_{\delta_x}(X_t=\mathbf 0)$ is non-increasing in $t$, and by $(3)$, we know that $v_t\in bp\mathscr B^\phi_E$ for $t$ large enough.
	Therefore, we have $v\in bp\mathscr B^\phi_E\subset L^2(E,m)$.
	Taking $f=V_s(\theta \mathbf 1_E)$ in \eqref{eq:FKPP_mild} and letting $\theta\to\infty$, by monotonicity and $(2')$, we have that, there is an $s_0>0$ such that
\begin{equation}\label{eq:extinction-vt-phi}
	\int_0^t \langle \Psi_0 v_{r+s},\phi^*\rangle_mdr
	=\langle v_s,\phi^*\rangle_m-\langle v_{t+s},\phi^*\rangle_m,
	\quad s\geq s_0, t\geq 0.
\end{equation}
	Letting $s\to\infty$, by monotonicity, we have
\[
	\int_0^t \langle \Psi_0v, \phi^* \rangle_m dr
	= t \langle \Psi_0 v, \phi^* \rangle_m
	= \langle v, \phi^* \rangle_m - \langle v, \phi^* \rangle_m
	= 0.
\]
	Since $\phi^*$ is strictly positive on $E$, we must have $\Psi_0(v) = 0, m\text{-a.e.}$.
	This, with \eqref{eq:kernMeanSemGroup}, implies that $S_t\Psi_0 (v)\equiv 0$ for any $t>0$.
	By $(1')$, we know that $S_t v_s(x)$ take finite value for $s$ large enough.
	Letting $s\to\infty$ in the \eqref{eq:reason-for-asp2'}, by monotonicity, we have
\[
	v(x)
	=S_tv(x)-\int_0^t S_{t-r}\Psi_0(v)(x)dr=S_tv(x),
	\quad x\in E, t\geq 0,
\]
which says that the non-negative function $v$, if not identically 0,  is an eigenfunction  of $L$  corresponding to $\lambda=0$, where $L$ is the generator of the semigroups $(S_t)_{t \geq 0}$.
	Since $v\in L^2(E,m)$, by the uniqueness of the eigenfunction   in  $L^2(E,m)$ corresponding to $\lambda=0$,
there is a constant $c\in \mathbb R$, such that $v(x) = c\phi(x)$ for all $x\in E$.
	So with $\Psi_0 (v) \equiv 0, m\text{-a.e.},$ we must have $v\equiv 0$.
	Using the fact that $v_t(x)$ converges to $0$ pointwise, by monotonicity and \eqref{eq:assertion3}, we can verify the desired result $(5)$.
\par
	For $(5)\Rightarrow (6)$, note that,
by the definition of $v_t$,
	for any $\mu\in\mathcal M_f^\phi$, we have
\[\begin{split}
	&-\log\mathbf P_\mu\{\exists t>0,\text{ s.t. }X_t=0\}
	=\lim_{t\to \infty}(-\log\mathbf P_\mu(X_t = \mathbf 0))
	=\lim_{t\to\infty}\langle\mu, v_t\rangle
	= 0.
\end{split}\]
\par
	Finally, note that $(6)\Rightarrow (4)$ and $(4)\Rightarrow (2)$ are obvious.
\end{proof}

\subsection{Kolmogorov type result}
	Let $\{(X_t)_{t\geq 0}; (\mathbf P_\mu)_{\mu \in \mathcal M_f}\}$ be the $(\xi,\psi)$-superprocess introduced in Subsection \ref{sec: Main results} which satisfies Assumptions \ref{asp:1'} and \ref{asp:2'} and \ref{asp:3}.
	In this subsection, we will give a proof of Theorem \ref{thm:Kolmogorov-type-of-theorem}. Thanks to Lemma \ref{lem:discuss-of-assumption2'}, we know that each of the statements in \ref{lem:discuss-of-assumption2'} is true.
	In particular, $v_t(x)/\phi(x)$ converges to $0$ uniformly in $x\in E$.

\begin{lem}\label{lem:Kolmogorov-1}
	Under Assumptions  \ref{asp:1'}, \ref{asp:2'} and \ref{asp:3}, we have
\[
	\sup_{x\in E}\Big | \frac{v_t(x)}{\langle v_t, \phi^*\rangle_m \phi(x)} - 1\Big |
	\xrightarrow[t\to\infty]{} 0.
\]
\end{lem}
\begin{proof}
	We use an argument similar to that used in \cite{Powell2016An-invariance} for critical branching diffusions.
Fix a non-trivial $\mu \in \mathcal M_f^\phi$, and let the spine
 immigration $\{(\xi_t)_{t\geq 0}, (Y_t)_{t\geq 0}, \mathbf n; \dot {\mathbf P}_\mu\}$ be given by Theorem \ref{thm:spinDec}.
For any $t > 0$, we have
\begin{equation}\label{eq:vt-and-Y}\begin{split}
	&\langle \mu,\phi\rangle \dot {\mathbf P}_\mu [(Y_t(\phi))^{-1}]
	\overset{\eqref{eq: consisten of dotPphiT}}{=}
	\langle \mu,\phi\rangle \mathbf P^{(\phi,T)}_\mu [(Y_t(\phi))^{-1}]
	\\&\overset{\text{Theorem \ref{prop:sizBiasNMeas}}}{=} \langle \mu,\phi\rangle \mathbb N_\mu^{w_t(\phi)} [ (w_t(\phi))^{-1}]
	= \mathbb N_{\mu}\{w_t(\phi)>0\}
	= \lim_{\lambda \to \infty} \mathbb N_\mu[1-e^{-\lambda w_t(\phi)}]
	\\&\overset{\text{Campbell's formula}}{=} \lim_{\lambda \to \infty} (-\log \mathbf P_\mu[e^{-\lambda X_t(\phi)}])
	= -\log \mathbf P_{\mu} \{ X_t = \mathbf 0\} \\
	& \overset{\eqref{eq: extinction probability with vt}}{=} \langle \mu,v_t \rangle.
\end{split}\end{equation}
	Taking $\mu = \delta_x$ in \eqref{eq:vt-and-Y}, we get $v_t(x)/\phi(x)=\dot{\mathbf P}_{\delta_x}[(Y_t(\phi))^{-1}]$.
	Taking $\mu = \nu$, we get $\langle v_t, \phi^*\rangle_m = \dot {\mathbf P}_{\nu} [(Y_t(\phi))^{-1}]$.
	Therefore, to complete the proof, we only need to show that
\[
	\sup_{x\in E}\Big | \frac {\dot{\mathbf P}_{\delta_x}[(Y_t(\phi))^{-1}]} {\dot {\mathbf P}_\nu [(Y_t(\phi))^{-1}]}-1\Big|
	\xrightarrow[t\to\infty]{} 0.
\]
	For any Borel subset $G\subset (0,t]$, define
\[
	Y^G_t
	:= \int_{G\times \mathcal W} w_{t-s} \mathbf n(ds,dw).
\]
	Then we have the following decomposition of $Y$:
\begin{equation}\label{eq:decomposition-on-Y}
	Y_t
	= Y^{(0,t_0]}_t + Y^{(t_0,t]}_t,
	\quad 0 < t_0 < t < \infty.
\end{equation}
	It is easy to see, from the construction and the Markov property of the spine immigration $\{Y,\xi; \dot {\mathbf P}\}$, that for any $0 < t_0 < t < \infty$,
\[
	\dot{\mathbf P} [(Y_t^{(t_0,t]}(\phi))^{-1}|\mathscr F^\xi_{t_0}]
	= \dot{\mathbf P}_{\delta_{\xi_{t_0}}}[(Y_{t-t_0}(\phi))^{-1}]
	= (\phi^{-1}v_{t-t_0})(\xi_{t_0}).
\]
	Therefore, we have
\[\begin{split}
	\dot{\mathbf P}_\nu[(Y_t^{(t_0,t]}(\phi))^{-1}]
	= \dot{\mathbb P}_{\nu}[(\phi^{-1}v_{t-t_0})(\xi_{t_0}) ]
	= \langle v_{t-t_0},\phi^* \rangle_m
\end{split}\]
	and
\begin{equation}
\label{eq:Yt0t}
	\dot{\mathbf P}_{\delta_x}[(Y_t^{(t_0,t]}(\phi))^{-1}]
	= \dot{\mathbb P}_x[(\phi^{-1}v_{t-t_0})(\xi_{t_0}) ]
	=  \int_E  \dot{q}(t_0,x,y)(\phi^{-1}v_{t-t_0})(y) m(dy).
\end{equation}
	By the decomposition \eqref{eq:decomposition-on-Y}, we have
\begin{equation}\label{eq:vt-equation}\begin{split}
	\phi^{-1}v_t(x)
	&= \dot {\mathbf P}_{\delta_x} [(Y_t(\phi))^{-1}]\\
	&= \dot {\mathbf P}_\nu [(Y^{(t_0,t]}_t(\phi))^{-1}] + \big( \dot {\mathbf P}_{\delta_x} [(Y^{(t_0,t]}_t(\phi))^{-1}] - \dot {\mathbf P}_\nu [(Y^{(t_0,t]}_t(\phi))^{-1}] \big) \\
	&\quad + \big( \dot{\mathbf P}_{\delta_x}[(Y_t(\phi))^{-1} - (Y^{(t_0,t]}_t(\phi))^{-1}] \big)\\
	&=: \langle v_{t-t_0},\phi^* \rangle_m + \epsilon_x^1(t_0,t) +\epsilon_x^2(t_0,t).
\end{split}\end{equation}
	Suppose that $t_0 >1$, and let $c,\gamma > 0$ be the constants in \eqref{eq:IU}, we have
\begin{equation}\label{eq:epsilon-1}\begin{split}
	|\epsilon_x^1(t_0,t)|
	& = \big| \dot {\mathbf P}_{\delta_x} [(Y^{(t_0,t]}_t(\phi))^{-1}] - \dot {\mathbf P}_\nu [(Y^{(t_0,t]}_t(\phi))^{-1}] \big| \\
	& = \big|  \int_E  \dot{q}(t_0,x,y)(\phi^{-1}v_{t-t_0})(y) m(dy) - \langle v_{t-t_0},\phi^* \rangle_m \big|\\
	& \leq \int_{y\in E} \big| \dot{q}(t_0,x,y) - (\phi\phi^*)(y) \big| (\phi^{-1}v_{t-t_0})(y) m(dy)\\
	& \leq ce^{-\gamma t_0}\langle v_{t-t_0},\phi^* \rangle_m .
\end{split}\end{equation}
	We also have
\begin{equation}\label{eq:epsilon-2}\begin{split}
	|\epsilon_x^2(t_0,t)|
	&= \big| \dot{\mathbf P}_{\delta_x}[(Y_t(\phi))^{-1} - (Y^{(t_0,t]}_t(\phi))^{-1}] \big| \\
	&= \dot{\mathbf P}_{\delta_x}[Y_t^{(0,t_0]}(\phi)\cdot (Y_t(\phi))^{-1}\cdot (Y^{(t_0,t]}_t(\phi))^{-1}]\\
	&\leq \dot{\mathbf P}_{\delta_x}[\mathbf 1_{Y_t^{(0,t_0]}(\phi)>0}\cdot (Y^{(t_0,t]}_t(\phi))^{-1}]\\
	&= \dot{\mathbf P}_{\delta_x} \big[\dot{\mathbf P}_{\delta_x}[\mathbf 1_{Y_t^{(0,t_0]}(\phi)>0}|\mathscr F^\xi_{t_0}] \cdot \dot{\mathbf P}_{\delta_x}[ (Y^{(t_0,t]}_t(\phi))^{-1}|\mathscr F^\xi_{t_0}] \big].
\end{split}\end{equation}
	Notice that, by Campbell's formula, one can verify that
\[
	\dot{\mathbf P}_{\delta_x}[e^{-\langle Y_t^{(0,t_0]},\theta \mathbf 1_E\rangle}|\mathscr F^\xi_{t_0}]
	= e^{-\int_0^{t_0}\psi'_0(\xi_s,V_{t-s}(\theta\mathbf 1_E)(\xi_s))ds}.
\]
	Letting $\theta \to \infty$ we have
\[
	\dot {\mathbf P}_{\delta_x} [ \mathbf 1_{Y_t^{(0,t_0]}=\mathbf 0} | \mathscr F^\xi_{t_0}]
	= e^{-\int_0^{t_0}\psi'_0(\xi_s,v_{t-s}(\xi_s))ds}.
\]
	We also have
\[\begin{split}
	\psi'_0(x,v_{t-s}(x))
	&= 2\alpha(x)v_{t-s}(x) +\int_{(0,\infty)} (1-e^{-yv_{t-s}(x)})y\pi(x,dy)\\
	&\leq \big( 2\alpha (x)+\int_{(0,\infty)}y^2\pi(x,dy) \big) v_{t-s}(x)\\
	&= A(x) v_{t-s}(x) \leq \| A\phi\|_\infty \|\phi^{-1}v_{t-s}\|_\infty.
\end{split}\]
	Therefore
\begin{equation}\label{eq:firstpart-of-Y}
	\dot{\mathbf P}_{\delta_x}[\mathbf 1_{Y_t^{(0,t_0]}\neq \mathbf 0}|\mathscr F^\xi_{t_0}]
	= 1-e^{-\int_0^{t_0}\psi'_0(\xi_s,v_{t-s}(\xi_s))ds}
	\leq t_0\| A\phi\|_\infty \|\phi^{-1}v_{t-t_0}\|_\infty.
\end{equation}
	Plugging \eqref{eq:firstpart-of-Y} into \eqref{eq:epsilon-2}, using \eqref{eq:Yt0t} and letting $c,\gamma > 0$ be the constants in \eqref{eq:IU}, we have that
\begin{equation}\label{eq:epsilon-2-final}\begin{split}
	|\epsilon_x^2(t_0,t)|
	& \leq t_0\| A\phi\|_\infty \|(\phi^{-1}v_{t-t_0})\|_\infty \dot{\mathbf P}_{\delta_x}[ (Y^{(t_0,t]}_t(\phi))^{-1}|\mathscr F^\xi_{t_0}] \\
	& \leq t_0\|  A \phi\|_\infty\|(\phi^{-1}v_{t-t_0}) \|_\infty \int_{E} \dot{q} (t_0,x,y)(\phi^{-1}v_{t-t_0})(y) m(dy)\\
	& \leq t_0\| A\phi\|_\infty \| \phi^{-1}v_{t-t_0}\|_\infty (1+ce^{-\gamma t_0}) \langle v_{t-t_0},\phi^* \rangle_m.
\end{split}\end{equation}
	Combining \eqref{eq:vt-equation}, \eqref{eq:epsilon-1} and \eqref{eq:epsilon-2-final}, we have that
\begin{equation}\label{vts-inequality}\begin{split}
	\Big|\frac{\phi^{-1}v_t(x)}{\langle v_{t-t_0},\phi^* \rangle_m}-1 \Big|
	&\leq \frac{|\epsilon_x^1(t_0,t)|}{\langle v_{t-t_0},\phi^* \rangle_m} + \frac{|\epsilon_x^2(t_0,t)|}{\langle v_{t-t_0},\phi^* \rangle_m}\\
	&\leq ce^{-\gamma t_0} +t_0\| A\phi\|_\infty \| \phi^{-1}v_{t-t_0}\|_\infty (1+ce^{-\gamma t_0}).
\end{split}\end{equation}
	Since we know from
	Lemma \ref{lem:discuss-of-assumption2'}(5)
	that $\| \phi^{-1}v_t\|_\infty\to 0$ when $t\to\infty$, there exists a map $t\mapsto t_0(t)$ such that,
\[
	t_0(t)
	\xrightarrow[t\to\infty]{} \infty;
	\quad t_0(t)\| \phi^{-1}v_{t-t_0(t)}\|_\infty
	\xrightarrow[t\to\infty]{} 0.
\]
	Plugging this choice of $t_0(t)$ back into \eqref{vts-inequality}, we have that
\begin{equation}\label{eq:k1}
	\sup_{x\in E}\Big|\frac{\phi^{-1}v_t(x)}{\langle v_{t-t_0(t)},\phi^* \rangle_m}-1 \Big|
	\xrightarrow[t\to\infty]{} 0.
\end{equation}
	Now notice that
\begin{equation}\label{eq:k2}\begin{split}
	\Big |\frac {\langle v_t, \phi^*\rangle_m} {\langle v_{t-t_0(t)} , \phi^*\rangle_m} - 1 \Big |
	&\leq \int \Big | \frac{\phi^{-1}v_t(x)}{\langle v_{t-t_0(t)} , \phi^*\rangle} - 1 \Big| \phi \phi^*(x) m(dx)\\
	&\leq \sup_{x\in E}\Big|\frac{\phi^{-1}v_t(x)}{\langle v_{t-t_0(t)},\phi^* \rangle_m}-1 \Big|
	\xrightarrow[t\to\infty]{} 0.
\end{split}\end{equation}
	Finally, by \eqref{eq:k1}, \eqref{eq:k2} and the property of uniform convergence,
\[
	\sup_{x\in E}\Big|\frac{\phi^{-1}v_t(x)}{\langle v_{t},\phi^* \rangle_m}-1 \Big|
	\xrightarrow[t\to\infty]{} 0,
\]
	as desired.
\end{proof}
\begin{lem}\label{lem:Kolmogorov-2}
	Under Assumptions \ref{asp:1'}, \ref{asp:2'} and \ref{asp:3}, we have
\[
	\frac{1}{t\langle v_t,\phi^*\rangle_m}
	\xrightarrow[t\to\infty]{} \frac{1}{2}\langle  A\phi,\phi\phi^*\rangle_m.
\]
\end{lem}
\begin{proof}
	We use an argument similar to that used in \cite{Powell2016An-invariance} for critical branching diffusions.
	According to \cite{RenSongZhang2015Limit}, we have that, for any $x\in E$ and $z\geq 0$,
\[\begin{split}
	R(x,z)
	:=\psi_0(x,z)-\frac{1}{2} A(x)z^2
	\leq e(x,z)z^2,
\end{split}\]
	where
\[
	e(x,z)
	:=\int_{(0,\infty)}y^2\big(1\wedge \frac{1}{6}yz\big)\pi(x,dy)
	\leq  A(x).	
\]
	By monotonicity, we have that
\begin{equation}\label{equation:reason2}
	e(x,z)
	\xrightarrow[z\to 0]{} 0,
	\quad x\in E.
\end{equation}
	Taking $b(t):=\langle v_t,\phi^*\rangle_m$ and writing $l_t(x):=v_t(x)-b(t)\phi(x)$, Lemma \ref{lem:Kolmogorov-1} says that,
\begin{equation}\label{equation:reason1}
	\sup_{x\in E}\Big|\frac{l_t(x)}{b(t)\phi(x)}\Big|
	\xrightarrow[t\to\infty]{} 0.
\end{equation}
	Now, taking $s_0>0$ as in \eqref{eq:extinction-vt-phi},
	we have that $t\mapsto b(t)$ is differentiable on the set
\[
	C
	=\{t> s_0: \text{the function}~ t \mapsto \langle\Psi_0(v_t),\phi^*\rangle_m~ \text{is continuous at}~ t \}
\]
	and that
\begin{equation}\label{equation:b(t)}\begin{split}
	\frac{d}{dt}b(t)
	&= -\langle\Psi_0(v_t),\phi^*\rangle_m
	= -\big\langle \frac{1}{2} A \cdot v_t^2+R (\cdot,v_t(\cdot)),\phi^*\big\rangle_m \\
	&= -\big\langle \frac{1}{2}  A \cdot \big(b(t)\phi+l_t\big)^2+R (\cdot,v_t(\cdot)),\phi^*\big\rangle_m\\
	&= -b(t)^2\big[\frac{1}{2} \langle  A\phi,\phi \phi^*\rangle_m+g(t)\big],
	\quad t\in C,
\end{split}\end{equation}	
	where
\[\begin{split}
	g(t)
	&= \Big\langle \frac{l_t}{b(t) \phi}, A\phi^2\phi^*\Big\rangle_m + \frac{1}{2}\Big\langle \Big(\frac{l_t}{b(t) \phi}\Big)^2, A\phi^2\phi^*\Big\rangle_m + \Big\langle \frac{R(\cdot,v_t(\cdot))}{b(t)^2 \phi^2},\phi^2\phi^*\Big\rangle_m \\
	&=: g_1(t) + g_2(t) + g_3(t).
\end{split}\]
	From \eqref{equation:reason1}, we have $g_1(t)\to 0$ and $g_2(t)\to 0$ as $t\to\infty$.
	From
\[\begin{split}
	\frac{R(x,v_t(x))}{b(t)^2 \phi(x)^2}
	\leq \frac{e(x,v_t(x))\cdot v_t(x)^2}{ b(t)^2\phi(x)^2}
	= e(x,v_t(x))\Big(1+\frac{l_t(x)}{ b(t) \phi(x)}\Big)^2,
\end{split}\]
	using \eqref{equation:reason1}, \eqref{equation:reason2}, Lemma \ref{lem:discuss-of-assumption2'} (5) and the dominated convergence theorem 
	($e(x,v_t(x))$ is dominated by $ A(x)$), we conclude that $g_3(t)\to 0$ as $t\to\infty$.
\par
	Finally, from \eqref{equation:b(t)} we can write
\begin{equation}\label{eq:derivative-of-b(t)-1}
	\frac{d}{dt} \Big(\frac{1}{b(t)}\Big)
	= -\frac{d b(t)}{b(t)^2dt}
	= \frac{1}{2}\langle  A\phi,\phi\phi^*\rangle_m + g(t),
	\quad t\in C.
\end{equation}
	Notice that, since the function $t\mapsto \langle\Psi_0(v_t),\phi^*\rangle_m$ is non-increasing in $t$, 
		the complement of $C$ has at most countably many elements.
	Therefore, using \eqref{eq:extinction-vt-phi} and \eqref{eq:derivative-of-b(t)-1}, one can verify that $t\mapsto \frac{1}{b(t)}$ is absolutely continuous on the interval $[s_0,t_0]$ as long as $s_0$ and $t_0$ are large enough.
	This allows us to integrate \eqref{eq:derivative-of-b(t)-1} on the interval $[s_0,t_0]$ with respect to the Lebesgue measure, and get that
\[
	\frac{1}{b(t_0)}
	= \frac{1}{b(s_0)} + \frac{1}{2}\langle A\phi,\phi\phi^*\rangle_m(t_0-s_0) + \int_{s_0}^{t_0} g(s)ds,
	\quad \text{for } 0\leq s_0\leq t_0 \text{ large enough}.
\]
	Dividing by $t_0$ and letting $t_0\to\infty$ in the above equation, we have
\[
	\frac{1}{b(t)t}
	\xrightarrow[t\to\infty]{} \frac{1}{2}\langle  A\phi,\phi\phi^*\rangle_m
\]
	as desired.
\end{proof}
\begin{proof}[Proof of Theorem \ref{thm:Kolmogorov-type-of-theorem}]
	For $\mu \in \mathcal M^\phi_f$, from Lemma \ref{lem:discuss-of-assumption2'}.(5) we know that
\begin{equation}\label{eq:kol-1}
	\langle \mu ,v_t\rangle
    = \int_E v_t(x) \mu(dx)
	= \int_E \frac{v_t(x)}{\phi(x)} \phi(x)\mu(dx)
	\xrightarrow[t\to\infty]{} 0.
\end{equation}
	From Lemma \ref{lem:Kolmogorov-1} we know that
\begin{equation}\label{eq:kol-2}
     \frac {\langle\mu, v_t\rangle}{ \langle v_t,\phi^*\rangle_m}
	= \int_E \frac{v_t(x)}{\langle v_t, \phi^* \rangle_m \phi(x)}\phi(x)\mu(dx)
	\xrightarrow[t\to\infty]{} \langle \mu,\phi\rangle.
\end{equation}
	It then follows from \eqref{eq:kol-1}, \eqref{eq:kol-2} and Lemma \ref{lem:Kolmogorov-2} that
\[\begin{split}
	t\mathbf P_{\mu}(X_t\neq \mathbf 0)
	&= t (1-e^{-\langle \mu, v_t \rangle})
	= t \langle v_t,\phi^*\rangle\frac{\langle \mu,v_t\rangle }{\langle v_t,\phi^*\rangle_m} \frac {1-e^{-\langle \mu, v_t \rangle}} {\langle \mu, v_t \rangle}\\
	&\xrightarrow[t\to\infty]{} \frac{\langle \mu,\phi\rangle} {\frac{1}{2}\langle  A \phi,\phi \phi^*\rangle_m},
	\quad x\in E.
\end{split}\]
\end{proof}

\subsection{Yaglom type result}
	Let $\{(X_t)_{t\geq 0}; (\mathbf P_\mu)_{\mu \in \mathcal M_f}\}$ be the $(\xi,\psi)$-superprocess introduced in Subsection \ref{sec: Main results} which satisfies Assumptions \ref{asp:1'} and \ref{asp:2'} and \ref{asp:3}.
	In this subsection, we will give a proof of Theorem \ref{thm:Yaglom-type-theorem}.

Slutsky's theorem is used quite often to  prove convergence in law of two components, in which one contributes to the limit, and the other one is negligible. The following proposition says that under $\dot {\mathbf P}_\mu$, 
the weighted  mass  $Y_t(\phi)$ coming off spine, normalized by $t$, converges to a Gamma distribution as $t\to\infty$.

\begin{prop}
\label{prop:yaglTheorSpinImmigr}
	Suppose that Assumptions \ref{asp:1'}, \ref{asp:2'}  and \ref{asp:3} hold.
	Suppose that $\mu \in \mathcal M_f^\phi$. Let $\{(\xi_t)_{t\geq 0}, (Y_t)_{t\geq 0}, \mathbf n; \dot {\mathbf P}_\mu\}$ be the spine immigration given by Theorem \ref{thm:spinDec}.
Then $W_t:= \frac{Y_t(\phi)}{t}$ converges weakly to a Gamma distribution
	$\Gamma(2,c_0^{-1})$ with $c_0 := \frac {1} {2} \langle \phi  A, \phi \phi^* \rangle_m$.
\end{prop}
\begin{proof}
We only have to prove that
\[
	\dot{\mathbf P}_\mu[e^{-\theta W_t}]
	\xrightarrow[t\to\infty]{}\frac{1}{(1 + c_0 \theta)^2},
	\quad \theta \geq 0,\mu\in \mathcal M^\phi_f.
\]
 First we consider the case when $\mu = \delta_x$ for an arbitrary $x\in E$.
	To simplify  notation, for all $x\in E,\theta\geq 0$ and $t\geq 0$, we write
\[\begin{split}
	J(x,\theta,t)
&:=(\phi A)(x)\dot{\mathbf P}_{\delta_{x}}[e^{-\theta W_t}]\widetilde{\mathbf P}_{x}[e^{-X_t(\frac{\theta\phi}{t})}],\\
	J_0(x,\theta,t)
&:=(\phi A)(x)\dot{\mathbf P}_{\delta_x}[e^{-\theta W_t}]
\end{split}\]
	and
\[\begin{split}
	M(x,\theta,t)
	&:=\Big|\frac{1}{(1+c_0\theta)^2}-\dot{\mathbf P}_{\delta_x}[e^{-\theta W_t}]\Big|.
\end{split}\]
	\emph{Step 1.  We will show that}
\begin{equation}\label{eq:yaglTheorSpinImmigrStep1}
	\dot{\mathbf P}_{\delta_x}[e^{-\theta W_t}]
	=\dot{\mathbf P}_{\delta_x}[ e^{-\int_0^1 du\int_0^\theta d\rho\cdot J(\xi_{ut},\rho(1-u),t(1-u))} ].
\end{equation}
	In fact, we have
\[
	\frac{\partial}{\partial \theta}\dot{\mathbf P}_{\delta_x}[e^{-\theta W_t}|\xi]
	= -\dot{\mathbf P}_{\delta_x}[W_te^{-\theta W_t}|\xi],
	\quad t\geq 0,\theta \geq 0.
\]
	Applying Lemma \ref{lem:key-lemma} with $K(dr)=\delta_t(dr)$ and $f_t=\frac{\theta\phi}{t}$, for each $\theta \geq 0$, we have
\[\begin{split}
	-\frac{\partial}{\partial \theta}\log \dot{\mathbf P}_{\delta_x}[e^{-\theta W_t}|\xi]
	&=\frac{\dot{\mathbf P}_{\delta_x} [W_t e^{-\theta W_t}|\xi]}{\dot{\mathbf P}_{\delta_x}[e^{-\theta W_t}|\xi]}\\
	&=\frac{1}{t}\int_0^t  (A\phi)(\xi_s)\dot{\mathbf P}_{\delta_{\xi_s}}[e^{-(\theta \frac{t-s}{t})W_{t-s}}]\widetilde{\mathbf P}_{\xi_s}[e^{-X_{t-s}(\frac{\theta\phi}{t})}]ds\\
	&=\int_0^1 J(\xi_{ut},\theta(1-u),t(1-u)) du.
\end{split}\]
	Integrating both sides of the above equation yields that
\[
	-\log \dot{\mathbf P}_{\delta_x}[e^{-\theta W_t}|\xi]
	=\int_0^1 du\int_0^\theta J(\xi_{ut},\rho(1-u),t(1-u)) d\rho,
\]
	which implies \eqref{eq:yaglTheorSpinImmigrStep1}.
\par
	\emph{Step 2. We will show that}
\begin{equation}\label{eq:yaglTheorSpinImmigrStep2}
	\int_0^1 du\int_0^\theta (J_0-J)(\xi_{ut},\rho(1-u),t(1-u)) d\rho
	\xrightarrow[t\to\infty]{L^2(\dot{\mathbf P}_{\delta_x})} 0,\quad \theta\geq 0.
\end{equation}
	To get this result, we will apply Lemma \ref{lem:ergodic} with
\begin{equation}\label{eq:F}\begin{split}
	F(x,u,t)
	&:=\int_0^\theta d\rho\cdot (J_0-J)(x,\rho(1-u),t(1-u))\\
	&=\int_0^\theta d\rho\cdot  (A\phi)(x)\dot{\mathbf P}_{\delta_{x}}[e^{-\rho(1-u)W_{t(1-u)}}]\widetilde{\mathbf P}_{x}[1-e^{-X_{t(1-u)}(\frac{\rho\phi}{t})}].
\end{split}\end{equation}
	Firstly note that $F(x,u,t)$ is bounded by $\theta\|\phi A\|_\infty$ on $E\times[0,1]\times[0,\infty)$.
	Secondly note that $F(x,u,t)\xrightarrow[t\to\infty]{} 0$ for each $x\in E$ and $u\in[0,1]$, since $|J_0-J|$ is bounded by $\|\phi A\|_\infty$ and
\[\begin{split}
	\big|(J_0-J)(x,\theta,t)\big|
	&=  (A\phi)(x)\dot{\mathbf P}_{\delta_{x}}[e^{-\theta W_{t}}]\widetilde{\mathbf P}_{x}[1-e^{-X_t(\frac{\theta\phi}{t})}]\\
	&\leq  (A\phi)(x)\widetilde{\mathbf P}_{x}(X_t\neq \mathbf 0) \\
	&=  (A\phi)(x)\frac{2\alpha(x)\mathbf P_{\mathbf 0}(X_t\neq \mathbf 0)+\int_{(0,\infty)}y^2\mathbf P_{y\delta_x}(X_t\neq \mathbf 0)\pi(x,dy)}{2\alpha(x)+\int_{(0,\infty)}y^2\pi(x,dy)}\\
	&\xrightarrow[t\to\infty]{} 0, \quad x\in E,\theta\geq 0.
\end{split}\]
	Therefore, we can apply Lemma \ref{lem:ergodic} with $F(x,u,t)$ given by \eqref{eq:F}, and get \eqref{eq:yaglTheorSpinImmigrStep2}.
\par
	\emph{Step 3. We will show that}
\begin{equation}\label{eq:yaglTheorSpinImmigrStep3}
	\frac{1}{(1+c_0\theta)^2}
	= \lim_{t\to\infty} \dot {\mathbf P}_{\delta_x} \Big[e^{- \int_0^1 du \int_0^\theta d\rho \frac{ (A\phi)(\xi_{ut})}{(1+c_0\rho(1-u))^2} }\Big],
	\quad \theta\geq 0.
\end{equation}
	By elementary calculus, the following map
\[
	(x,u)
	\mapsto\int_0^\theta \frac{ (A\phi)(x)}{(1+c_0\rho(1-u))^2} d\rho
	= \frac{ (A\phi)(x)\theta}{1+c_0\theta(1-u)}
\]
	is bounded by $\theta\| A\phi\|_\infty$ on $E\times[0,1]$.
	According to Lemma \ref{lem:ergodic}, we have that
\[\begin{split}
	\int_0^1 du \int_0^\theta \frac{ (A\phi)(\xi_{ut})}{\big(1+c_0\rho(1-u)\big)^2} d\rho
	&\xrightarrow[t\to\infty]{L^2(\dot{\mathbf P}_{\delta_x})} \int_0^1 \big\langle \frac{\theta A\phi}{1+c_0\theta(1-u)},\phi\phi^* \big\rangle_m du\\
	&= \langle  A\phi,\phi\phi^*\rangle_m\int_0^1 \frac{\theta}{1+c_0\theta(1-u)}du \\
	&=2\log(1+c_0\theta).
\end{split}\]
	Therefore, by the bounded convergence theorem, we get \eqref{eq:yaglTheorSpinImmigrStep3}.
\par
	\emph{Step 4. We will show that}
\begin{equation}\label{eq:yaglTheorSpinImmigrStep4}
	M(x,\theta)
	:=\limsup_{t\to\infty}M(x,\theta,t)=0,
	\quad x\in E,\theta\geq 0.
\end{equation}
	In fact,
\begin{equation}\label{eq:separate-M-into-3-parts}
	M(x,\theta,t)
	\leq I_1+I_2+I_3,
\end{equation}
	where
\[
	I_1
	:=\Big|\frac{1}{(1+c_0\theta)^2}- \dot {\mathbf P}_{\delta_x} \big[e^{- \int_0^1 du \int_0^\theta \frac{ (A\phi)(\xi_{ut})}{[1+c_0\rho(1-u)]^2}d\rho }\big]\Big|
	\xrightarrow[t\to\infty]{\text{by \eqref{eq:yaglTheorSpinImmigrStep3}}} 0,
\]
\[\begin{split}
	I_2
	&:=\Big| \dot {\mathbf P}_{\delta_x} [e^{- \int_0^1 du \int_0^\theta  \frac{ (A\phi)(\xi_{ut})}{(1+c_0\rho(1-u))^2} d\rho}]-\dot{\mathbf P}_{\delta_x}[e^{-\int_0^1 du\int_0^{\theta} J_0(\xi_{ut},\rho(1-u),t(1-u)) d\rho}]\Big|\\
	&\leq\dot{\mathbf P}_{\delta_x}\Big[\int_0^1du\int_0^\theta (A\phi)(\xi_{ut})M(\xi_{ut},\rho(1-u),t(1-u))d\rho \Big]\\
	&=\int_0^1du\int_0^\theta d\rho \int_{E} \dot{q}(ut,x,y) (A\phi)(y)M(y,\rho(1-u),t(1-u)) m(dy),
\end{split}\]
	and by \eqref{eq:yaglTheorSpinImmigrStep1} and \eqref{eq:yaglTheorSpinImmigrStep2},
\[\begin{split}
	I_3
	&:=\big|\dot{\mathbf P}_{\delta_x}[e^{-\int_0^1 du\int_0^{\theta} J_0(\xi_{ut},\rho(1-u),t(1-u)) d\rho }]-\dot{\mathbf P}_{\delta_x}[e^{-\theta W_t}]\big|\\
	&=\big|\dot{\mathbf P}_{\delta_x}[e^{-\int_0^1 du\int_0^{\theta} J_0(\xi_{ut},\rho(1-u),t(1-u)) d\rho }]-\dot{\mathbf P}_{\delta_x}[e^{-\int_0^1 du\int_0^{\theta} J(\xi_{ut},\rho(1-u),t(1-u)) d\rho }]\big|\\
	&\leq \dot{\mathbf P}_{\delta_x}\Big[\Big|\int_0^1du\int_0^\theta (J_0-J)(\xi_{ut},\rho(1-u),t(1-u))d\rho\Big|\Big]
	\xrightarrow[t\to\infty]{} 0.
\end{split}\]
	Therefore, taking $\limsup_{t\to\infty}$ in \eqref{eq:separate-M-into-3-parts}, by the reverse Fatou's lemma, we get
\begin{equation}\label{eq:the-key-inequiality}
	M(x,\theta)
	\leq \int_0^1du\int_0^\theta \langle  A\phi M(\cdot,\rho(1-u)),\phi\phi^*\rangle_m d\rho,
	\quad x\in E,\theta\geq 0.
\end{equation}
	Integrating with respect to the finite measure $(A\phi\phi\phi^*)(x)m(dx)$ yields that
\[
	\langle  A\phi M(\cdot,\theta),\phi\phi^*\rangle_m
	\leq \langle A\phi,\phi\phi^*\rangle_m\int_0^1du\int_0^\theta \langle  A\phi M(\cdot,\rho(1-u)),\phi\phi^*\rangle_m d\rho,
	\quad \theta\geq 0.
\]
	According to \cite[Lemma 3.1]{RenSongSun2018A-2-spine}, this inequality implies that $\langle A\phi M(\cdot,\theta),$ $\phi\phi^*\rangle_m$ $ =0$ for each $ \theta\geq 0$.
	This and \eqref{eq:the-key-inequiality} imply \eqref{eq:yaglTheorSpinImmigrStep4}, which completes the proof when $\mu=\delta_x$.
\par
	Finally, for any $\mu \in \mathcal M^\phi_f$, since
\[\begin{split}
	\langle \mu,\phi\rangle\dot {\mathbf P}_\mu [e^{-\theta W_t}]
	&=\langle\mu,\phi\rangle\mathbb N^{w_t(\phi)}_\mu[e^{-\theta \frac{w_t(\phi)}{t}}]
	=\mathbb N_\mu [w_t(\phi)e^{-\theta \frac{w_t(\phi)}{t}}]\\
	&=\int_E\mu(dx)\mathbb N_x[w_t(\phi)e^{-\theta\frac{w_t(\phi)}{t}}]
	=\int_E \mu(dx) \phi(x) \dot{\mathbf P}_{\delta_x}[e^{-\theta W_t}],
\end{split}\]
	we have that, by the bounded convergence theorem,
\[
	\big|\dot {\mathbf P}_\mu[e^{-\theta W_t}] - \frac{1}{(1+c_0\theta)^2}\big|
	\leq \int_E \big|\dot {\mathbf P}_{\delta_x}[e^{-\theta W_t}]-\frac{1}{(1+c_0\theta)^2}\big| \frac{\phi(x)\mu(dx)}{\langle\mu,\phi\rangle}
	\xrightarrow[t\to\infty]{} 0,
\]
	as desired.	
\end{proof}

The following lemma says that, conditional on survival up to time $t$, the weighted and normalized mass $t^{-1}X_t(\phi)$ 
(weighted  by $\phi$, and normalized  by $t$) 
has a limit distribution which is exponential with explicit parameter. Later we will consider limit of 
$t^{-1}X_t(f)$  with a general $f\in bp\mathscr B^\phi_E$.

\begin{lem}\label{lem:Yaglom-for-phi}
	Suppose that Assumptions \ref{asp:1'}, \ref{asp:2'} and \ref{asp:3} hold. Let $\mu \in \mathcal M_f^\phi$.
	Then it holds that $\{t^{-1}X_t(\phi);\mathbf P_\mu(\cdot | X_t\neq\mathbf 0)\}$
   converges weakly to 
   an exponential distribution
	$\operatorname{Exp}(c_0^{-1})$ with $c_0 := \frac{1}{2}\langle\phi A,\phi\phi^*\rangle_m$.
\end{lem}

\begin{proof}
 We only have to show that
\[\begin{split}
	\mathbf P_\mu[e^{-\theta t^{-1}X_t(\phi)} |X_t\neq\mathbf 0]
	&\xrightarrow[t\to\infty]{}\frac{1}{1+c_0\theta},
	\quad \theta\geq 0, \mu\in \mathcal M_f^\phi.
\end{split}\]
	Notice that, by Lemma \ref{lem:discuss-of-assumption2'}(6), we have
\[
	\{t^{-1}X_t(\phi);\mathbf P_\mu\}
	\xrightarrow[t\to\infty]{law} 0.
\]
	Therefore, by Theorem \ref{thm:spinDec} and Proposition \ref{prop:yaglTheorSpinImmigr}, we have
\[
	\mathbf P^M_\mu[e^{-\theta t^{-1}X_t(\phi)}]
	= (\mathbf P_\mu \otimes \dot {\mathbf P}_\mu)[e^{-\theta t^{-1}(X_t+Y_t)(\phi)}]
	\xrightarrow[t\to\infty]{}\frac{1}{(1+c_0\theta)^2}.
\]
	Also notice that, by elementary calculus
\[
	\frac{1-e^{-\theta u}}{u}
	=\int_0^\theta e^{-\rho u} d\rho,
	\quad u> 0.
\]
	From Theorem \ref{thm:spinDec} and Lemma \ref{lem:Y-is-immortal} we know that $\mathbf P^M_\mu(X_t=\mathbf 0)=0$.
	Therefore by the bounded convergence theorem, we have
\[\begin{split}
	\mathbf P^M_\mu \Big[\frac{1-e^{-\theta t^{-1}X_t(\phi)}}{t^{-1}X_t(\phi)}\Big]
	&= \mathbf P^M_\mu\Big[\int_0^\theta e^{-\rho t^{-1}X_t(\phi) }d\rho\Big]
	= \int_0^\theta \mathbf P^M_\mu [ e^{-\rho t^{-1}X_t(\phi) }] d\rho\\
	&\xrightarrow[t\to\infty]{} \int_0^\theta \frac{1}{(1+c_0\rho)^2}d\rho
	= c_0^{-1} (1-\frac{1}{1+c_0\theta}).
\end{split}\]
	Hence, by Theorem \ref{thm:Kolmogorov-type-of-theorem} we have
\[\begin{split}
	\mathbf P_\mu[1-e^{-\theta t^{-1}X_t(\phi)} |X_t\neq \mathbf 0]
	&= \mathbf P_\mu(X_t\neq \mathbf 0)^{-1}\mathbf P_\mu[(1-e^{-\theta t^{-1}X_t(\phi)})\mathbf 1_{X\neq \mathbf 0}]\\
	&= \mathbf P_\mu(X_t\neq \mathbf 0)^{-1}\mathbf P_\mu\Big[(1-e^{-\theta t^{-1}X_t(\phi)})\frac{X_t(\phi)}{X_t(\phi)}\Big]\\
	&= (t\mathbf P_\mu(X_t\neq \mathbf 0))^{-1} \langle\mu,\phi\rangle \mathbf P^M_\mu \Big[\frac{1-e^{-\theta t^{-1}X_t(\phi)}}{t^{-1}X_t(\phi)}\Big]\\
	&\xrightarrow[t\to\infty]{}1-\frac{1}{1+c_0\theta},
\end{split}\]
	which completes the proof.
\end{proof}

Now we consider  limit of $t^{-1} X_t(f)$ with general weight $f\in bp\mathscr B^\phi_E$ . The main idea is to use the following decomposition for $f$:
$f(x)=\langle \phi^*,f\rangle_m \phi(x)+\tilde f(x), x\in E$. The following lemma says that $\tilde f$ has no contribution to the limit, and then we can easily get that the conditional limit of $t^{-1} X_t(f)$ as $t\to\infty$ is the contribution of $\langle \phi^*,f\rangle_mt^{-1}X_t(\phi)$, which is known from Lemma \ref{lem:Yaglom-for-phi}.

\begin{lem}\label{lem:general-lemma}
	Suppose that Assumptions \ref{asp:1'}, \ref{asp:2'} and \ref{asp:3} hold.
	If $\tilde f\in b\mathscr B^\phi_E$ satisfies $\langle \tilde f, \phi^*\rangle = 0$, then we have, for any $\mu \in \mathcal M^\phi_f$,
\[\begin{split}
	\big\{ t^{-1} X_t(\tilde f) ; \mathbf P_\mu(\cdot|X_t \neq \mathbf 0)\big\}
	&\xrightarrow[t\to\infty]{} 0,
	\quad\text{in probability}.
\end{split}\]
\end{lem}
\begin{proof}
	If we can show that $\mathbf P_\mu\big[\big(t^{-1}X_t(\tilde f)\big)^2|X_t \neq \mathbf 0\big]\xrightarrow[t\to\infty]{} 0$, then the desired result follows by the Chebyshev's inequality
\[
	\mathbf P_\mu \big( | t^{-1} X_t(\tilde f) | \geq \epsilon \big | X_t \neq \mathbf 0\big)
	\leq \epsilon^{-2}\mathbf P_\mu \big[ \big(t^{-1} X_t(\tilde f)\big)^2 \big | X_t \neq \mathbf 0 \big].
\]
	By Proposition \ref{prop:covanrance} we have that
\begin{equation}\label{eq:general-lemma-1}\begin{split}
	&\mathbf P_\mu\big[\big(t^{-1}X_t(\tilde f)\big)^2\big|X_t \neq \mathbf 0\big]
	= t^{-2} \mathbf P_\mu (X_t\neq \mathbf 0)^{-1}\mathbf P_\mu\big[X_t(\tilde f)^2\mathbf 1_{X_t\neq \mathbf 0}\big] \\
	&\quad= t^{-1} \mathbf P_\mu (X_t\neq \mathbf 0)^{-1} \Big(\frac{\langle \mu,S_t\tilde f\rangle^2}{t} + \langle \mu,\phi \rangle\dot{\mathbb P}_{\mu}\big[(\phi^{-1}\tilde f)(\xi_t)\frac{1}{t}
	\int_0^t A(\xi_s)\cdot (S_{t-s} \tilde f)(\xi_s)ds\big]\Big).
\end{split}\end{equation}
	Letting $c,\gamma > 0$ be the constants in \eqref{eq:IU}, we know that
\begin{equation}\label{eq:ll1}\begin{split}
	| S_t\tilde f (x) - \langle \phi^* , \tilde f \rangle_m \phi(x)|
	&= \Big | \int_{E} \big(q(t,x,y) - \phi(x)\phi^*(y)\big) \tilde f (y) m(dy) \Big | \\
	&\leq \int_{E} \big|\frac{q(t,x,y)}{\phi(x)\phi^*(y)} - 1 \big| \cdot |\phi(x)\phi^*(y) \tilde f (y) | m(dy) \\
	& \leq ce^{-\gamma t} \phi(x) \|\phi^{-1}\tilde f\|_\infty \int_{E} (\phi\phi^*)(y) m(dy)\\
	& \xrightarrow[t\to\infty]{} 0,
	\quad x\in E.
\end{split}\end{equation}
Therefore, by the dominated convergence theorem,
\[
	\langle \mu,S_t\tilde f\rangle
	\xrightarrow[t\to \infty]{} \langle \phi^*, \tilde f\rangle_m\langle \mu, \phi\rangle
	= 0.
\]
	Hence,	
\begin{equation}\label{eq:S_t}
	\frac{ \langle \mu,S_t\tilde f\rangle}{t}
	\xrightarrow [t\to\infty]{} 0,
	\quad x\in E.
\end{equation}
	By \eqref{eq:ll1} and Lemma \ref{lem:ergodic}, we know that
\[\begin{split}
\frac{1}{t}\int_0^t A(\xi_s) \cdot (S_{t-s} \tilde f)(\xi_s)ds
	&= \int_0^1 A(\xi_{ut})\cdot (S_{t-ut} \tilde f)(\xi_{ut})d{u}\\
    &\xrightarrow[t\to\infty]{L^2(\dot{\mathbb P}_x)} \int_0^1\langle A\phi,\phi\phi^*\rangle_m \langle\phi^*,\tilde f\rangle_m du
	= 0.
\end{split}\]
	Hence, by Lemma \ref{lem:1stMomSizBiasSupProc} and the bounded convergence theorem we have that
\begin{equation}\label{eq:int}\begin{split}
	&\big| \langle \mu , \phi \rangle \dot{\mathbb P}_{\mu}\big[(\phi^{-1}\tilde f)(\xi_t)\frac{1}{t}\int_0^t A(\xi_s)\cdot (S_{t-s} \tilde f)(\xi_s)ds\big]\big|\\
	&\quad\leq \int \mu(dx)\phi(x) \big|\dot{\mathbb P}_x \big[(\phi^{-1}\tilde f)(\xi_t)\frac{1}{t}\int_0^t A(\xi_s)\cdot (S_{t-s} \tilde f)(\xi_s)ds\big]\big| \\
	&\quad\leq \|\phi^{-1} \tilde f\|_\infty \cdot \int \mu(dx)\phi(x) \dot{\mathbb P}_{x}\Big[ \big| \frac{1}{t}\int_0^t A(\xi_s)\cdot (S_{t-s} \tilde f)(\xi_s)ds \big|^2 \Big]^{\frac{1}{2}}\\
	&\quad \xrightarrow[t\to\infty]{} 0.
\end{split}\end{equation}
	Finally, using  Theorem \ref{thm:Kolmogorov-type-of-theorem} and combining  \eqref{eq:general-lemma-1}, \eqref{eq:S_t} and  \eqref{eq:int},   we have that
\[
	\mathbf P_{\mu}\big[\big(t^{-1}X_t(\tilde f)\big)^2\big|X_t \neq \mathbf 0\big]
	\xrightarrow[t\to\infty]{} 0
\]
	as required.
\end{proof}
\begin{proof}[Proof of Theorem \ref{thm:Yaglom-type-theorem}]
	Define a function $\tilde f$ by
\begin{equation}\label{eq:yy1}
	\tilde f(x)
	:=f(x) - \langle \phi^*,f\rangle_m \phi(x),
	\quad x\in E.
\end{equation}
	It is easy to see that $\tilde f\in b\mathscr B^\phi_E$ and $\langle\tilde f,\phi^* \rangle_m = 0$.
	It then follows from Lemma \ref{lem:Yaglom-for-phi} that
\begin{equation}\label{eq:yy2}
	\big\{ t^{-1}X_t(\langle \phi^*,f\rangle_m \phi);\mathbf P_\mu(\cdot | X_t\neq \mathbf 0)\big\}
	\xrightarrow[t\to\infty]{law} \frac{1}{2}\langle \phi^*,f\rangle_m\langle \phi A, \phi\phi^*\rangle_m \mathbf e,
\end{equation}
	and from Lemma \ref{lem:general-lemma} that
\begin{equation}\label{eq:yy3}
	\big\{ t^{-1} X_t(\tilde f) ; \mathbf P_\mu(\cdot|X_t \neq \mathbf 0)\big\}
	\xrightarrow[t\to\infty]{\text{in probability}}0.
\end{equation}
	The desired result then follows from \eqref{eq:yy1}, \eqref{eq:yy2}, \eqref{eq:yy3} and Slutsky's theorem.
\end{proof}

\begin{rem}
	In the symmetric case, i.e. when $(S_t)$ are self-adjoint operators,  \eqref{eq:yy1} is exactly an $L^2$-orthogonal decomposition.
\end{rem}

\bigskip
\noindent
{\bf Acknowledgments:}
We thank the two referees for very helpful comments on the first version
of this paper.

\end{document}